    \documentclass{amsart}
\usepackage{amsmath,amsthm}
\usepackage{verbatim}
\usepackage{marvosym}
\usepackage{graphicx}
\usepackage{color}
\usepackage{epsfig}
\usepackage{rotating}
\usepackage{lscape}
\usepackage{mathrsfs, euscript, mathdesign}
\usepackage[normalem]{ulem}
\usepackage{cleveref}
    \usepackage{baskervald}
\usepackage[T1]{fontenc}

\usepackage{tikz}
\usetikzlibrary{cd}
\usepackage{thmtools}

\theoremstyle{plain}
\newtheorem{thm}{Theorem}[section]
\newtheorem{cor}[thm]{Corollary}

\newtheorem{prop}[thm]{Proposition}
\newtheorem{lemma}[thm]{Lemma}

\theoremstyle{definition}

\newtheorem{remark}[thm]{Remark}

 \newcommand{\Hom}{\ensuremath{{\mathrm{Hom}}}}

\newcommand{\Ox}{\ensuremath{\mathcal{O}}}

\newcommand{\Q}{\ensuremath{\mathbb{Q}}}

\newcommand{\Z}{\ensuremath{\mathbb{Z}}}
\newcommand{\C}{\ensuremath{\mathbb{C}}}

\renewcommand{\H}{\ensuremath{\mathbb{H}}}

\newcommand{\SL}{\text{SL}}
\newcommand{\PSL}{\text{PSL}}
\newcommand{\GL}{\text{GL}}
\newcommand{\tr}{\text{tr}}

\newcounter{nootje}
\newcommand{\mat}[4]{\left(\begin{array}{cc}#1 & #2 \\ #3 & #4 \end{array}\right)}

\begin{document}
\title[]{Symmetries and Detection of Surfaces by the Character Variety}

\author[Jay Leach and Kathleen L. Petersen]{J. Leach and K. L. Petersen}


\begin{abstract}
We extend Culler and Shalen's construction of detecting essential surfaces in 3-manifolds to 3-orbifolds.  We do so in the setting of the $\SL_2(\C)$ character variety, and following Boyer and Zhang in the $\PSL_2(\C)$ character variety as well.  We show that any slope detected on a canonical component of the $\mathrm{(P)SL}_2(\C)$ character variety of a one cusped hyperbolic 3-manifold with symmetries must be the slope of a symmetric surface.  As an application, we show that for each symmetric double twist knot there are slopes which are strongly detected on the character variety but not on the canonical component.

\end{abstract}
 
\maketitle


\section{Introduction}

Culler and Shalen \cite{MR683804} developed a beautiful theory showing  how to  associate surfaces in 3-manifolds to representations of their fundamental groups.  
To certain points and ideal points in the $\SL_2(\C)$ character variety of a 3-manifold  there are valuations that value negatively.  Such a valuation induces a non-trivial action on  a Bass-Serre tree.   One can associate  essential surfaces dual to such an action. We say such a  surface is {\em detected} by the action. 
 This theory was extended to the $\PSL_2(\C)$ setting by Boyer and Zhang \cite{MR1670053}. 
 We study how the existence of a symmetry of the manifold effects the  detection of surfaces. 

The $\SL_2(\C)$ character variety of a manifold $M$, $X(M)$, is the set of representations of $\pi_1(M)$ to $\SL_2(\C)$ up to trace equivalence, and the $\PSL_2(\C)$ character variety $\bar{X}(M)$ can be defined similarly. These character varieties are $\C$ algebraic sets.
An irreducible component of the $\SL_2(\C)$ character variety of a hyperbolic 3-manifold $M$ is called a  {\em canonical component}, and is written $X_0(M),$  if it contains the character of a discrete and faithful representation. Canonical components are often called geometric components because they encode a wealth of geometric information about the underlying manifold.  Thurston \cite{thurston} showed that all but finitely many representations coming from Dehn filling of a single cusp are on a canonical component. 
We will call a surface in $M$  {\em symmetric} if it is fixed set-wise by the group of orientation preserving isometries of $M$, and {\em non-symmetric} otherwise. 
In Section~\ref{mainsection} we prove the main goal of this paper.   \begin{restatable}{thm}{maintheorem}
\label{thm:maintheorem}
Let $M$ be a finite volume, orientable,  hyperbolic 3-manifold with a single cusp, and $G$ a subgroup of the  group of orientation preserving isometries of $M$ with the property that  the orbifold quotient $M/G$ has a flexible cusp. 
Let $v$ be a discrete valuation on a field $F$  such that there is a representation $\phi:\pi_1(M) \rightarrow \SL_2(F)$ with $v(\tr(\phi(\gamma)))<0$ for some $\gamma \in \pi_1(M)$. Then $v$ detects an essential surface in $M$.  If this representation $\phi$ is associated to a point (or ideal point)  on a canonical component $X_0(M)$ then $v$ detects a symmetric essential surface in $M$.
\end{restatable}
Essential surfaces are detected in a highly non-canonical way. A single point can detect multiple surfaces. This theorem shows that if $M$ has a symmetry, then any point (or ideal point)  on a canonical component  which detects an essential surface detects a symmetric essential surface.

 For affine points on the character variety, we consider any valuation. One example is the $\pi$-adic  valuation for a prime ideal $\pi$ associated to an algebraic non-integral (ANI) point. 
For ideal points we consider  the order of vanishing valuation.  An ideal point  is   a point in a smooth projective completion that is not in the character variety, a point at infinity. (See Section~\ref{section:PSLOrbs}.) 
 We will work with the $\PSL_2(\C)$ character variety for this proof. There are obstructions to lifting a $\PSL_2(\C)$ representation to an $\SL_2(\C)$ representation which make the $\PSL_2(\C)$ character variety is more widely applicable,  especially when studying  the quotient orbifold  $M/G$.

A key step in the proof of Theorem~\ref{thm:maintheorem} is our extension of Culler and Shalen's construction of essential surfaces in 3-manifolds from actions on trees \cite{MR683804, MR1886685} to the orbifold setting.  
 We show that this theory can be extended to orbifolds, in both the $\SL$ and $\PSL$ setting. 
In Section~\ref{section:PSLOrbs} we discuss the valuations for orbifold groups in both the $\SL$ and $\PSL$ setting, recalling how these valuations give rise to actions on trees.  And we prove the following.
\begin{restatable}{thm}{orbifoldmaintheorem}
\label{orbifoldTreeAction} 
Let $Q$ be a compact, orientable, irreducible 3-orbifold.  If $\pi_1^{orb}(Q)$ acts non-trivially and without inversions on a tree $T$, then there exists an essential $2$-suborbifold $F$ in $Q$ dual to this action.
\end{restatable}

A   similar result was proven independently by  Yokoyama   \cite{MR3501267}.  However, we require some specific properties for our applications to surface detection.  For example,   we assume that the orbifolds are orbifold irreducible and show that the constructed $2$-suborbifolds can always be made transverse to the singular set of the orbifold.
 This is necessary  to show that these $2$-suborbifolds lift to essential surfaces in covering manifolds, which we show in Lemma~\ref{suborbliftinglemma} in Section~\ref{section:lifting}.

A slope $r\in \Q \cup \{1/0\}$ is called a {\em boundary slope} if there is an essential surface $F$ in $M$ such that $\partial F \cap \partial M$ is a non-empty set of parallel simple closed curves on $\partial M$ of slope $r$.  A boundary slope $r$ is called {\em strict} if there is an essential surface $F$ in $M$  which is not a  fiber or semi-fiber such that  $\partial F \cap \partial M$ has slope $r$. 
It is known \cite[Proposition 1.2.7]{MR881270} that fibered and semi-fibered surfaces are never detected by ideal points of  components of the character variety which contain irreducible representations. 
  Let $x$ be a point in a smooth projective closure of $X(M)$. 
The boundary slope $r$   is said to be {\em detected} at $x$   if a surface with boundary slope $r$  is detected at $x$.    
We call a slope $r$ {\em strongly detected} by $x$ if $x$ detects no closed essential surface, and otherwise we say $r$ is {\em weakly detected}  by $x$.  (Some authors use the term detected exclusively for detection at an ideal point.)  
We will call a  boundary slope which is the slope of a symmetric  essential surface   {\em symmetric}, and we will call all other boundary  slopes {\em non-symmetric}. 
In \cite{ccgls} it was shown that the detected boundary slopes at ideal points are the slopes of the Newton polygon of the $A$-polynomial.

For knot complements in $S^3$ with the standard framing the boundary slope $0/1$ is detected by characters of abelian representations. This corresponds to the longitude slope, which is the slope of a Seifert surface.   
 Motegi \cite{MR953952} showed that there are closed graph manifolds that contain essential tori  not detected by ideal points of the  character variety.
Boyer and Zhang \cite[Theorem 1.8]{ MR1670053}  showed that there are infinitely many closed hyperbolic 3-manifolds
whose character varieties do not detect closed essential surfaces at ideal points. Schanuel
and Zhang \cite[Example 17] {MR1835066} gave an example of a closed hyperbolic 3-manifold with a closed essential
surface that is not   detected by an ideal point but is ANI-detected. 
They also constructed a family of graph manifolds with boundary slopes that are not strongly detected at ideal points, but are weakly detected.
Chesebro and Tillmann \cite{MR2395254} used mutants to construct  examples of (cylindrical) knot complements containing strict
boundary slopes which are weakly detected at ideal points,  but not strongly detected at ideal points.
Chesebro \cite{MR3073906} also demonstrated a connection between  module structures on the coordinate ring of an irreducible
component of the character variety and the detection of closed essential surfaces.  Casella, Katerba, and Tillmann \cite{CKT} used this to show that there are closed essential surfaces in hyperbolic knot complements which are not detected by ideal points of the character variety.  These surfaces are also not ANI-detected.

The following corollary about boundary slopes follows immediately from Theorem~\ref{thm:maintheorem}.
\begin{cor}
Let $M$ be a finite volume, orientable,  hyperbolic 3-manifold with a single cusp,   such that  the orbifold quotient $M/G$ has a flexible cusp.  Any boundary  slope detected on $X_0(M)$ is a symmetric slope.
\end{cor}
\noindent
We will prove the following in Section~\ref{mainsection}.
\begin{cor}
\label{maincor}
Let $M$ be a finite volume, orientable,  hyperbolic 3-manifold with a single cusp,  such that  the orbifold quotient $M/G$ has a flexible cusp.    There are at least two distinct symmetric essential surfaces in $M$ with different  slopes.
\end{cor}
\noindent
In Section~\ref{examplesection}, we use these results to prove a general statement for hyperbolic two-bridge knots. 
\begin{cor}\label{cor:twobridge}
Every hyperbolic two-bridge knot complement contains at least two symmetric essential surfaces with different boundary slopes. \end{cor}

\noindent
In Section~\ref{examplesection}  we  study the symmetric double twist knots, an infinite   family of two-bridge knots. Let $K_n$ be the double twist knot with two twist regions with $n$ full twists, as defined in Section~\ref{examplesection}, and $X(K_n)$ the $\SL_2(\C)$ character variety of its complement. 
We show that every boundary slope is detected by some ideal point on the character variety, and all symmetric slopes are detected by  ideal points on a canonical component.  However, for each $n$, there is a boundary slope not detected on the canonical component.

\begin{restatable}{thm}{symmetricdoubletwistthm}
\label{mainresultTheoremdetectedslopesstrict} Let $K_n$ be a hyperbolic symmetric double twist knot. 
The ideal points on $X(K_n)$  detect the slopes  $0$, $-8n+2$ and $-4n$; these are all of the boundary slopes   in $S^3-K_n$.
The ideal points on the canonical component $X_0(K_n)$  detect slopes $0$ and $-8n+2$;  these are all of the symmetric boundary slopes of $S^3-K_n$.
\end{restatable}

\subsection{Organization}

In Section~\ref{section:orbifolds} we recall the definition and construction of an orbifold and the orbifold fundamental group.  We collect facts about 3-orbifolds and  suborbifolds.  In Section~\ref{actionontreesection} we prove Theorem~\ref{orbifoldTreeAction} and then prove Lemma~\ref{suborbliftinglemma}, which shows that we can lift a detected $2$-suborbifold in $M/G$ to an essential surface in $M$.  We review the construction of the $\SL_2(\C)$ and $\PSL_2(\C)$ character varieties in Section~\ref{section:charactervarieties}.  We also prove Theorem~\ref{thm:symmetiesandcanonicalcomponent} which says that a canonical component is fixed by the induced action of a symmetry.  In Section~\ref{section:inducedmaps}  we  discuss induced maps on the character variety.  We also discuss $\SL_2(\C)$ and $\PSL_2(\C)$ Culler Shalen theory for manifolds and orbifolds in Section~\ref{section:PSLOrbs}.   In Section~\ref{mainsection} we use this and  Theorem~\ref{orbifoldTreeAction} to prove  Theorem~\ref{thm:maintheorem}. Section~\ref{examplesection}  is devoted to  the double twist knots.

\section{Orbifolds }\label{section:orbifolds}

 We refer the reader to \cite{MR1778789} for a comprehensive treatment of orbifolds.    Associated to the orbifold $Q$, we let $X_Q$ denote the {\em underlying space}, and $\Sigma(Q)$  \ be the {\em singular locus} of $Q$. The singular locus consists of the points  $x\in X_Q$ with non-trivial {\em local group} $G_x$.

An \textit{orbifold covering map} $p:P\to Q$ is a continuous map of underlying spaces from $X_P$ to $X_Q$ with the following property.   Each point $x\in X_Q$ has a neighborhood $U$ which necessarily is diffeomorphic to $\tilde{U}/D$ for some affine patch $\tilde{U}$ and finite group $D$ of diffeomorphisms of $\tilde{U}$.  We require that   each component of $p^{-1}(U)$ is isomorphic to $\tilde{U}/D'$ for some $D'<D$.
Every orbifold $Q$ has a universal covering $p:\tilde{Q}\to Q$ such that for every orbifold covering map $f:X\to Q$ there is an orbifold covering map $p':\tilde{Q}\to P$ such that $f\circ p':\tilde{Q}\to Q$ is also an orbifold covering map.
The \textit{orbifold fundamental group}, $\pi_1^{orb}(Q)$, is the group of deck transformations of the orbifold universal cover $\tilde{Q}$.
  An orbifold is  \textit{(very) good} if it is (finitely) covered by a manifold.  Otherwise, it is  \textit{bad}.

 If $M$ is a manifold and a group $G$ acts on $M$ properly discontinuously, then $M/G$ has an induced orbifold structure.  
 In the context of  Theorem~\ref{thm:maintheorem}  we will consider  orbifolds which are the quotient of compact, orientable, irreducible 3-manifolds by a finite symmetry group. All such orbifolds are themselves compact, orientable, and orbifold irreducible. Therefore, we will restrict our attention to orbifolds that are compact, orientable, and orbifold irreducible.
For our application to  Theorem~\ref{thm:maintheorem} we only need to consider good orbifolds, but the proof of Theorem~\ref{orbifoldTreeAction} does not require the orbifold to be good.

The assumption that the orbifold $Q$ is orientable gives a useful description of the singular set.
\begin{thm}\cite[Theorem 2.5]{MR1778789}\label{thm:singularset}
Let $Q$ be a orientable $3$-orbifold.  Then the underlying space $X_Q$ is an orientable 3-manifold, and the singular $\Sigma(Q)$  set consists of a trivalent graph where  the three edges at any vertex have orders $(2,2,k)$, $(2,3,3)$, $(2,3,4)$, or $(2,3,5)$.
\end{thm}

The 2-dimensional orbifolds which
are 2-spheres with 3 cone points and cone angles $(\pi, \pi/2, \pi/2)$, $(\pi, 2\pi /3, \pi/3)$ and
$(2\pi/3, 2\pi /3, 2\pi /3)$ are the Euclidean  turnovers $S^2(2,4,4)$, $S^2(2,3,6)$, and $S^2(3,3,3)$, respectively.
The 2-dimensional orbifold which is 2-sphere with 4 cone points all of cone angle $\pi$ is the pillowcase, $S^2(2, 2, 2, 2)$.
If $X$ is an orientable, non-compact finite volume hyperbolic 3-orbifold, then a cusp
of $X$ has the form $Q \times [0, \infty)$, where $Q$ is a Euclidean orbifold.  In fact (see \cite{MR1291531}) $Q$ is either a torus, a pillowcase or a turnover. 
The cusp is said to be
{\em rigid} if $Q$ is a Euclidean turnover, otherwise it is called {\em flexible}.

\subsection{Suborbifolds}

We begin with the manifold setting.  Let $M$ denote an orientable compact, irreducible, hyperbolic 3-manifold.  
Let $F$ be a (not necessarily connected) surface in $M$.
A {\em bicollaring} of  $F$ is a homeomorphism $h$ of $F\times  [ -1,1 ]$ onto a neighborhood of $F$ in $M$ such that $h(x,0)=x$ for every $x\in F$ and $h(F\times [ -1,1 ])\cap \partial M=h(\partial F\times [ -1,1 ])$. A surface $F$ is \textit{bicollared} if $F$ admits a bicollaring.  
A surface $F$ in $M$ is {\em  boundary parallel}  if there exist a homotopy of $F$ that takes $F$ into the boundary of $M$.

A surface $F$ is  \textit{essential} if it has the following properties:
$F$ is bicollared;
every component of $F$ is $\pi_1$-injective in $M$;
no component of $F$ is a $2$-sphere;
no component of $F$ is boundary parallel; and
$F$ is nonempty.

In the orbifold setting we have the following.
A  $2$-suborbifold $F$ in a $3$-orbifold $Q$ is \textit{bicollared} if $X_F$ is a bicollared as an embedded surface in the manifold $X_Q$.
 An \textit{orbifold disc} (or \textit{ball}) is the quotient of a $2$-disc (or $3$-ball) by a finite group.
 A {\em spherical} 2-orbifold is a quotient of a 2-sphere by a finite group. 
A $2$-suborbifold $F$ in a $3$-orbifold $Q$ is called \textit{orbifold incompressible} if for any component $P$ of $F$, $P$ has the following properties, where $\chi$ is the orbifold Euler characteristic (see \cite{MR1778789}):
\begin{itemize}
\item[$(a)$]  $\chi(P)>0$ and    $P$ doesn't bound an orbifold ball in $Q$. \\
 and
\item[$(b)$]  $ \chi(P)\leq 0$ and  any $1$-suborbifold in $P$ that bounds a orbifold disc in $Q-P$ also bounds an orbifold disc in $P$.
\end{itemize}
 We will call incompressible $2$-suborbifold \textit{essential} if it 
  is bicollared and has no boundary parallel or spherical components.
 We call a   $3$-orbifold  \textit{orbifold irreducible} if it contains no bad $2$-suborbifold or essential spherical $2$-suborbifold.

\section{ Dual $2$-suborbifolds to Actions on Trees}\label{actionontreesection}

The action of a group $\Gamma$ on a tree $T$ is said to be {\em trivial} if there is a vertex of $T$ fixed by the entire group $\Gamma$, otherwise the action is called non-trivial.  The group $\Gamma$ acts {\em without inversions} if $\Gamma$ does not reverse the orientation of any invariant edge. 
We will now assume that  $Q$ is a compact, orientable, and orbifold-irreducible $3$-orbifold  and  $\pi_1^{orb}(Q)$ acts non-trivially and without inversions on a tree $T$.
Our proof of Theorem~\ref{orbifoldTreeAction}  is modeled the work in  \cite{MR942518},  \cite{MR881270} and \cite{MR1886685} for surfaces in 3-manifolds, and uses   ideas from \cite{MR1065604}.

We now  outline  the proof.  First, we   make the action simplicial. In Lemma~\ref{treeproof1}  we obtain  a triangulation of the orbifold $Q$ that lifts to a triangulation of the universal orbifold cover $\tilde{Q}$.  In Lemma~\ref{treeproof2} we use this triangulation to construct from $\tilde{Q}$ a $\pi_1^{orb}(Q)$-equivariant, simplicial map $\tilde{f}$ from $\tilde{Q}$ into $T$.   For  $E$ the set of midpoints of edges in $T$, we look at the sets $\tilde{F}=\tilde{f}^{-1}(E)\subset \tilde{Q}$ and $F=p(\tilde{F}) \subset Q$.  We show that they are bicollared, and after altering $\tilde{f}$ by homotopy  show that we can make $F$ an essential $2$-suborbifold.  We will also show that we can take the suborbifold to be transverse to the singular set, which   will allow us to lift the suborbifold to  an essential surface  in a covering space.


Let $\tilde{Q}$ be the universal covering space of $Q$ with covering map $p:\tilde{Q}\rightarrow Q$.

\begin{lemma}\label{treeproof1}
There is a  triangulation $\Delta$ of $X_Q$ such that the singular set of $Q$ is contained in the $1$-skeleton of $\Delta$.  This triangulation lifts to a triangulation $\tilde{\Delta}$ of $\tilde{Q}$.

\end{lemma}

\begin{proof}

 Let $\Delta$ be any triangulation of $X_Q$, the underlying space of $Q$.  We will   show that we can perform homotopy and barycentric subdivision   to obtain a triangulation $\Delta'$ of $X_Q$ with the property that the singular set of $Q$ is contained in the $1$-skeleton of $\Delta'$.  As a result,  this triangulation $\Delta'$ will lift to a triangulation of $\tilde{Q}$. 
 
By Theorem~\ref{thm:singularset} due to \cite{MR1778789} the singular set of $Q$ is a finite trivalent graph.  Therefore,  we can perform  a homotopy on $\Delta$ to make the singular set transverse to every $2$-cell.  We   then perform barycentric subdivision on the triangulation   until the interior of every $3$-cell   contains at most $1$ vertex of the singular set of $Q$ and every $2$-cell intersects the singular set at most once (while maintaining the transverse property). 
 Therefore,  any singular set that intersects the interior of a $3$-cell either intersects as a line segment  with endpoints on different faces, or as three line segments meeting at a vertex (again with endpoints on different faces).  We now can perform barycentric subdivision    so that   the singular set is  contained in the new $1$ and $2$ cells of the triangulation.  After homotopy, it is contained in the $1$-skeleton of the barycentric subdivision.  
\end{proof}


Given a triangulation $\Delta$ as in Lemma~\ref{treeproof1}, let  $\tilde{\Delta}^i$ represent the $i$-skeleton of $\tilde{\Delta}$.   Let $S^i$ be a complete system of orbit representatives for the action of $\pi_1^{orb}(Q)$ on $\tilde{\Delta}^i$. 
 For each  $s\in S^0$, define the subgroup stabilizing $s$ as  
\[ \Gamma_s=\{\gamma\in\pi_1^{orb}(Q) \mid  \gamma \cdot s = s \}.\] 
Then $\Gamma_s$ is a subgroup of the local group $G_{p(s)}$, and is a finite group.   If $p(s)$ is not in the singular locus of $Q$, then $\Gamma_s$ is trivial.  
 Let $E$ denote the set of midpoints of edges of $T$.

We begin with a map $h_0:S^0\rightarrow T^0$ with the property that $h_0(s)$ is stabilized by $\Gamma_s$.  This local group is finite and so there is at least one vertex in $T$ stabilized by the whole group. Therefore the condition that $h_0(s)$ is stabilized by $\Gamma_s$ can be easily achieved. 
For example, when  $ Q$ is hyperbolic  then the  space $\tilde{Q}=\H$ is contractible,  so we can embed $T$ in $\tilde{Q}$ and take $h:\tilde{Q}\rightarrow T$ to be a deformation retraction. Since $T$ is necessarily an infinite tree, in practice, we can  alter $h$ by homotopy so that $h_0$ sends $S^0$ to $T^0$.

\begin{lemma}\label{treeproof2}
Given a map $h_0:S^0\rightarrow T^0$ such that for each point $s\in S^0,$  $h_0(s)$ is stabilized by $\Gamma_s$,   there is a $\pi_1^{orb}(Q)$-equivariant map $\tilde{f}:\tilde{Q}\rightarrow T$ that restricts to $h_0$ and is transverse to $E$.

\end{lemma}

The proof will proceed as follows. 
We will  define a continuous and  $\pi_1^{orb}$-equivariant map $\tilde{f_i}$ from the $i$-skeleton of $\tilde{Q}$ to $T$ and then extend it to the $(i+1)$-skeleton until we have a map $\tilde{f}=\tilde{f}_3$ defined $f$ on $\tilde{Q}$. The final equivariant  map $\tilde{f}$ is not unique as many  choices are involved in the construction of this map.

\begin{proof}
We begin by constructing $\tilde{f}_0$. The set  $S^0\subset \tilde{\Delta}^0\subset \tilde{Q}$ and every orbit for the action of $\pi_1^{orb}$ on $\tilde{\Delta}^0$ intersects $S^0$ in precisely one point. 
Because the action of any finite group on a tree is trivial, $\Gamma_s$ acts trivially on $T$ and   fixes at  least one  vertex of $T$.

Given $s'\in \tilde{\Delta}^0-S^0$  then $s'=\gamma\cdot s$ for some $\gamma \in \pi_1^{orb}$  and we extend $h_0$ by defining  $\tilde{f}_0(s')=\gamma\cdot h^0(s)$. 
Given such an $s'$, if $s'=\gamma \cdot s = \gamma' \cdot s$ then $\gamma^{-1}   \gamma' \cdot s = s$ and so $ \gamma^{-1}   \gamma' \in \Gamma_s <G_{p(s)}$.  By construction, $h_0(s)$ is fixed by $\gamma^{-1}  \gamma' $.  From this it follows that $\gamma\cdot h_0(s)=\gamma' \cdot h_0(s)$ and indeed $\tilde{f}_0$ is well-defined. 
The map $\tilde{f}_0$ is unique because the $\pi_1^{orb}$ equivariant condition necessitates that $\tilde{f}_0(\gamma \cdot s)= \gamma \cdot h_0(s)$. 
We conclude that $h_0$ can be uniquely extended to a $\pi_1$-equivariant map $\tilde{f}_0$ from $\tilde{\Delta}^0$ to $T^0$.

Next we extend the  map $\tilde{f}_0$   to  $\tilde{f}_1: \tilde{\Delta}^1 \rightarrow T$.  We start by  choosing $S^1$.   For each  $1$-simplex $\sigma\in S^1$ let $h_{\sigma}:\sigma\to T$ be a continuous map that agrees with $\tilde{f}_0$ on $\partial\sigma$ and maps into the unique line segment in $T$ connecting the two points of $\tilde{f}_0(\partial\sigma)$. If both endpoints of $\sigma$ are mapped to the same vertex of $T$ by $\tilde{f}_0$, then $h_{\sigma}$ also maps every point in $\sigma$ to that vertex. Let $\tilde{f}_1$ be the unique $\pi_1^{orb}$-equivariant map from $S^1$ into $T$ given by 
\[ \tilde{f}_1(\gamma \cdot s) = \gamma \cdot h_{\sigma}(s)\]
for all $\gamma\in\pi^{orb}_1(Q)$, $\sigma\in S^1$, and $s\in\sigma$.
Note that this results in a simplicial map.

Again we need to check that this function is well-defined.  It agrees with $\tilde{f}_0$ on the boundary of every simplex, so it is enough to consider $s' \in \tilde{\Delta}^1-S^1$.   Assume that $s' = \gamma \cdot s = \gamma'' \cdot s''$ for some $s\in \sigma,$ a $1$-simplex in $\tilde{\Delta}^1$ and $s'' \in \sigma'',$ a $1$-simplex in $\tilde{\Delta}^1$. Then $\sigma = \sigma''$ because $S^1$ is a complete system of orbit representatives, and we have  $ \gamma \cdot s = \gamma '' \cdot s''$ so that $s = \gamma^{-1} \gamma'' \cdot s''$.  It follows that $\gamma^{-1} \gamma''$ sends $\sigma$ to itself. 
First, consider the case when $s\neq s''$.  Then the action of $\gamma^{-1} \gamma''$ on $\tilde{\Delta}^1$ interchanges the endpoints of $\sigma$, so the action of $\gamma^{-1} \gamma''$ must do so as well under $\tilde{f}_0$ since it is a $\pi_1^{orb}$-equivariant map.  Because $\pi_1^{orb}(Q)$ acts on $T$ without inversions,  both endpoints of $\sigma$ must be mapped to the same point, $x$,  in $T$, and so all of $\sigma$ is mapped to $x$  and necessarily $\gamma^{-1} \gamma''$ fixes $x$.  Therefore,  considering $s'=\gamma \cdot s$ we have 
\[ \tilde{f}_1(s') = \tilde{f}_1( \gamma \cdot s) = \gamma \cdot h_{\sigma}(s) = \gamma \cdot x.\] 
Similarly, with $s'= \gamma'' \cdot s''$ we have $\tilde{f}_1(s') = \gamma'' \cdot x$, which equals $\gamma \cdot x$ since $\gamma^{-1} \gamma ''$ fixes $x$. 
It now suffices to consider the case when $s=s''$, so that $s' = \gamma \cdot s = \gamma'' \cdot s$ and $\gamma^{-1} \gamma '' $ fixes $s$.   Since the action is without inversions, it cannot be the case that $\gamma^{-1} \gamma ''$ fixes $s$ only and interchanges the endpoints of the $1$-simplex $\sigma$,  so $\gamma^{-1} \gamma ''$ fixes $\sigma$ point-wise.  Therefore, $h_{\sigma}$ is the identity.  It follows that $\gamma \cdot h_{\sigma}(s)= \gamma'' \cdot h_{\sigma}(s)$, and so $\tilde{f}_1(s')$ is well-defined.

Now we extend this map to  $\tilde{f}_2:\tilde{\Delta}^2\to T$.    For any simplex $\sigma\in S^2$ we have a map $\tilde{f}_2|_{\partial\sigma}:\partial\sigma\to T$ by restricting to $\tilde{f}_1$.  Because $T$ is contractible, this map can be continuously extended to a map $h_{\sigma}:\sigma\to T$. 
  Now let the map $\tilde{f}_2:\tilde{\Delta}^2\to T$ be defined as 
  \[ \tilde{f}_2(\gamma \cdot s)= \gamma \cdot  h_{\sigma}(s)\text{ for all $\gamma\in\pi_1^{orb}(Q)$, $\sigma\in S^2$, and $s\in\sigma$.}\]
We have already shown that this is well-defined on the boundaries of every simplex.  Because every point in the interior of a simplex is not a lift of a singular point (those are all in the $1$-skeleton by construction) and $\pi_1^{orb}(Q)$ acts freely on all points of $\tilde{Q}$ that are not lifts of singular points, $\tilde{f}_2$ is well-defined.  By the simplicial approximation theorem \cite{MR1325242} $\tilde{f}_2$ can  be made simplicial.

The construction of $\tilde{f}=\tilde{f}_3$ is identical to that of $\tilde{f}_2$.    Note that $\tilde{f}$ is $\pi_1$-equivariant by construction.
We can perform a homotopy to ensure the transverse condition.
 \end{proof}


The next step is to use the map $\tilde{f}$ to construct a bicollared $2$-suborbifold in $Q$.  We will assume that we have a triangulation of $\tilde{Q}$ that is $\pi_1^{orb}(Q)$ invariant so that $\tilde{f}$ is simplicial.    
Let $\tilde{f}$  be as in Lemma~\ref{treeproof2}, and let $\tilde{F}=\tilde{f}^{-1}(E)\subset \tilde{Q}$, and $F=p(\tilde{F})\subset Q$.

\begin{lemma}\label{lemma:bicollared}
Consider an action of $\pi_1^{orb}(Q)$ on a tree $T$ that is non-trivial and without inversions.
The set $\tilde{F}$ is a two sided and bicollared surface in $\tilde{Q}$, and the set $F$ is a two sided and bicollared $2$-suborbifold in $Q$. 
\end{lemma}

\begin{proof}

Let $s\in T-T^0$ contained in edge $e$, and consider $P=\tilde{f}^{-1}(s)\subset \tilde{Q}$.  Let $\sigma$ be an $i$-simplex of $\tilde{Q}$ and consider $P \cap \sigma$.   If $\tilde{f}$ does not map $\sigma$  onto $e$ then $P\cap \sigma = \emptyset$.  Otherwise, if $\tilde{f}$ does map $\sigma$ onto $e$, no vertex of $\sigma$ is mapped to $s$ since $s$ is not a vertex and the action is simplicial. Therefore, the set $P \cap \sigma$ is an $(i-1)$-cell which is properly embedded in $\sigma$.  We see that $P$ intersects every simplex of $\tilde{Q}$ either trivially or in a properly embedded codimension 1 cell, and so $P$ is a 2-manifold. 
Since for each $s\in T-T^0$ the set $\tilde{f}^{-1}(s)$ is a properly embedded 2-manifold in $\tilde{Q}$, it follows that $\tilde{F}$ is a properly embedded 2-manifold in $\tilde{Q}$. Each point of $E$ is two sided in $T$ and therefore $\tilde{F}=\tilde{f}^{-1}(E)$  is two sided in $\tilde{Q}$ since $\tilde{f}$ is transverse to $E$. The surface $\tilde{F}$  is bicollared in $\tilde{Q}$ because $\tilde{F}$ is bicollared in each cell of $\tilde{\Delta}$ and the action of $T$ is without inversions.

 The set $E$ is $\pi_1^{orb}$ invariant, and $\tilde{f}$ is $\pi_1^{orb}$ equivariant, so $\tilde{F}$ is invariant under the  action of  $\pi_1^{orb}$  on $\tilde{Q}$ and therefore $F$ is a $2$-suborbifold.  Let  $\bar{p}:T\rightarrow T/\pi_1^{orb}(Q)$ be  the natural map.  The map $\tilde{f}$ induces a unique map $f:Q\rightarrow T/\pi_1^{orb}(Q)$ with the property that $f\circ p = \bar{p} \circ \tilde{f}$.  Since $\tilde{f}$ is transverse to $E$ it follows that $f$ is transverse to $\bar{p}(E)$.  Each point of $\bar{p}(E)$ is two sided in $T/\pi_1^{orb}(Q)$, and it follows that $F$ is two sided in $X_Q$.  Since the singular set $\Sigma(Q)$ is contained in the 1-skeleton of $\Delta$ and $\tilde{f}$ acts simplicially, it follows that $\tilde{F}$ is two sided in $Q$.

The singular set  $\Sigma(Q)$ consists of edges and vertices where 3 edges meet.  We have constructed a triangulation so that the singular set is contained in the 1-skeleton.   The set $\tilde{F}$ is transverse to the 1-skeleton of $\tilde{Q}$, and since $p^{-1}(\Sigma(Q))$ is contained in the 1-skeleton of $\tilde{Q}$ we conclude that  $\tilde{F}$ is transverse to $p^{-1}(\Sigma (Q))$.  
The image $F=p(\tilde{F})$ in $Q$ is a $2$-suborbifold.  Because of the transversality,  $F$ only contains a finite number of singular points, and  the bicollaring can be extended to all of $F$.  Thus $F$ is bicollared in $Q$. 
\end{proof} 

Following Culler and Shalen, we say that a bicollared 2-suborbifold $F$ in $Q$ is {\em dual} to an action of $\pi_1^{orb}(Q)$ on $T$ if is arises from this construction for some $\pi_1^{orb}(Q)$ equivariant map that is transverse to $E$. 


If $Q$ is a 3-orbifold and $F$ is a sub-orbifold then there is a natural map $i_*:\pi_1^{orb}(F) \rightarrow \pi_1^{orb}(Q)$ induced by inclusion $i:F\rightarrow Q$.  (And similarly for a component $C_i$ of $Q-F$.) See \cite[Page 167]{MR1065604}  for details.  
Many standard theorems about submanifolds carry through to the orbifold case, for example \cite[Lemma 3.10]{MR1065604}   if $F$ is compact, two sided and incompressible then the induced map is injective.  These $i_*$ maps are well-defined up to conjugation if no base point is specified.

\begin{lemma}\label{lemma:stabilizers}
Let $F$ be a dual 2-suborbifold to an action of $\pi_1^{orb}(Q)$ on a tree $T$.  Then 
\begin{enumerate}
\item\label{lemma:stabilizers:vertex} For each component $C_i$ of $Q-F$ the subgroup $i_*(\pi_1^{orb}(C_i))$ of $\pi_1^{orb}(Q)$ is contained in the stabilizer of some vertex of $T$.
\item\label{lemma:stabilizers:edge} For each component $F_j$ of $F$, the subgroup $i^*(\pi-1^{orb}(F_j))$ of $\pi_1^{orb}(Q)$ is contained in the stabilizer of some edge of $T$.
\end{enumerate}
\end{lemma}

\begin{proof}
Let $\tilde{f}:\tilde{Q} \rightarrow T$ be a $\pi_1^{orb}(Q)$-equivariant map transverse to the set $E$ of midpoints of edges of $T$ so that $F=f^{-1}(E)$.   Let $\Gamma_i$ be $i_*(\pi_1^{orb}(C_i) ) < \pi_1^{orb}(Q)$.   

By construction of $F$,  $\tilde{f}$ maps any connected lift of $C_i$ into  a component of $T-E$.
  The subgroup  $\Gamma_i$ stabilizes a component $\tilde{C}$ of $\tilde{Q}-\tilde{F}$. 
Since $\tilde{f}$ is $\pi_1^{orb}(Q)$ equivariant, $\Gamma_i$ stabilizes $\tilde{f}( \tilde{C} )$.  Since this is contained in a component of $T-E$, $\tilde{f}$  stabilizes the lone vertex in this component.

For the second item,  we argue as above and conclude that $i_*(\pi_1^{orb}(F_j))$ stabilizes a component of $T-T^0$, where $T^0$ is the vertex set.  Since the action is simplicial, and without inversions this group must stabilize a whole edge. 
\end{proof}

We want this construction to give an incompressible 2-suborbifold in $Q$, so the next step is to make $F$ orbifold incompressible.   This will complete the proof of Theorem~\ref{orbifoldTreeAction}.

 \orbifoldmaintheorem*

\begin{proof}
As before, we have a $\pi_1^{orb}(Q)$ equivariant map $\tilde{f}:\tilde{Q} \rightarrow T$ that is transverse to $E$ where $F=f^{-1}(E)$ and $\tilde{f}^{-1}(E)=p^{-1}(F)$ in $\tilde{Q}$.  If $F$ were empty  the $Q-F$ consists of all of $Q$ and by Lemma~\ref {lemma:stabilizers} (\ref{lemma:stabilizers:edge}) we conclude that $\pi_1^{orb}(Q)$ fixes a vertex of $T$, which by definition means that this action is trivial. Therefore, $F$ is non-empty.

From the construction thus far, and Lemma~\ref{lemma:bicollared} in particular, we may assume that $F$  is a bicollared 2-suborbifold.    We may assume that $F$ is not essential, and so either $F$ is not incompressible or $F$ has boundary parallel or spherical components.  We will first consider the case when $F$ is not incompressible. We will determine another $\pi_1^{orb}(Q)$-equivariant map $\tilde{f}':\tilde{Q}\rightarrow T$ which is also transverse to $E$ such that $(\tilde{f}')^{-1}(E) = p^{-1}(F')$ for a 2-suborbifold $F'$ dual to the action.

If $F$ is not incompressible, then there is a component $F_0$ of $F$  that contains a $1$-suborbifold that bounds an orbifold disc $D$ in $Q-F_0$ which does not bound an orbifold disc in $F_0$.  
If $D$ intersects more than one component of $F$ then we can chose an innermost intersection of $D$ with a component of $F$ to get a new compressing orbifold disc $D'$ for the component $F_0'$.  Thus we will assume that $D\cap F=\partial D\subset F_0$.

We may assume that $\partial D$ does not intersect $\Sigma(Q)$ by pushing  $\partial D$ off the singular points.
 Let $A$ be an annular neighborhood of $\partial D$ in $F_0$ (we choose $A$ such that it contains no singular points).  Let $D_1$ and $D_2$ be two parallel copies of $D$, whose boundaries are the two components of $\partial A$.  Note that the union of $D_1$ and $D_2$ and $A$ is a $2$-sphere in the underly space $X_Q$.  Then because $Q$ is orbifold irreducible the union of $D_1$ and $D_2$ and $A$ is a spherical suborbifold that bounds an orbifold ball in $Q$.

\begin{figure}
\begin{center}
\includegraphics[width=3in]{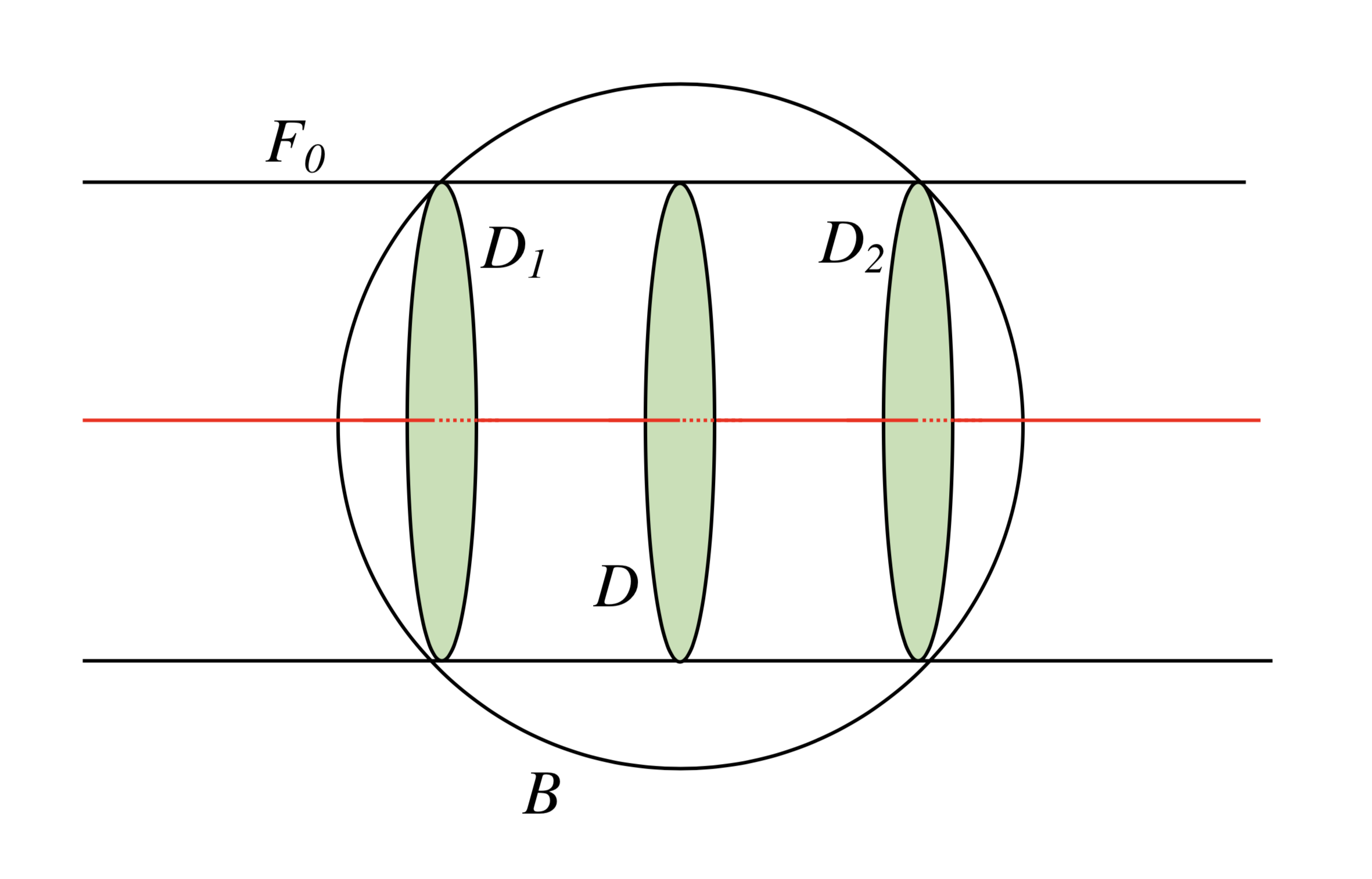} 
\caption{Compressing orbifold disk}
 \label{compressingSurgery}
\end{center}
\end{figure}

Let $B$ be a nice neighborhood of $D$ in $Q$ such that $B$ is an orbifold ball and meets $F$ along the boundary of the annulus $A$, so  $B\cap F_0 = \partial A$.   Because the singular locus is a trivalent graph in $Q$, we can perform a  homotopy so that $D$ does not contain a vertex of this graph and therefore choose  $B$ to intersect only one `edge' of the singular locus.  Therefore, $B$ is locally the quotient of a ball by a finite cyclic group.  

Then $B$ is the union of a solid torus $X^+$ and an orbifold ball $X^-$, where $X^+\cap X^- = A$.  We may take $D_1$ and $D_2$ to be properly embedded in $X^+$.  Now let $\tilde{B}$ be a component of $p^{-1}(B)$, and let $\tilde{A}$, $\tilde{D_1}$, $\tilde{D_2}$, $\tilde{X}^+$, and $\tilde{X}^-$ denote the inverse images of $A$, $D_1$, $D_2$, $X^+$, and $X^-$ in $\tilde{B}$.

Because $\tilde{f}$ is transverse to $E$ and $\tilde{A}=\tilde{B}\cap\tilde{F}^{-1}(E)$ and since $T$ is a tree, $\tilde{f}$ maps $\tilde{X}^+$ and $\tilde{X}^-$ to the closure of two different components of $T-E$.  Lets call these components $Y^+$ and $Y^-$ respectfully.

We are now going to construct a new $\pi_1$-equivariant map $\tilde{f}'$ by removing the compressing orbifold disk.  Let $\tilde{f}'$ be equal to $\tilde{f}$ for all of $\tilde{Q}-p^{-1}(B)$.  First define $\tilde{f}'$ to be constant on both $\tilde{D_1}$ and $\tilde{D_2}$ and equal to the unique point of $E$ that is in the intersection of $Y^+$ and $Y^-$.  Extend this map to the rest of $p^{-1}(D_1\cup D_2)$ by having $\tilde{f}'(\gamma \cdot x)=\gamma\cdot \tilde{f}'(x)$ for all $x\in p^{-1}(D_1\cup D_2)$ and $\gamma\in\pi_1(Q)$.  To see that this is well-defined notice that if $\gamma \cdot x =\gamma' \cdot x'$, then $x=\gamma^{-1}\gamma' \cdot x'$, meaning $\gamma^{-1}\gamma'$ stabilizes $\tilde{D_1}$ (or $\tilde{D_2}$) in $\tilde{Q}$, and  $\gamma \cdot \tilde{f}'(x)=\gamma' \cdot \tilde{f}'(x')$.  If $\gamma \cdot \tilde{f}'(x') \neq \gamma' \cdot \tilde{f}'(x')$, then $\gamma^{-1}\gamma'$ does not stabilize $Y^+ \cap Y^-$.  However, $\gamma^{-1}\gamma'$ stabilizes $\tilde{D_1}$ ($\tilde{D_2}$) and thus $\partial\tilde{D_1}$  ($\partial\tilde{D_2}$). Thus $\gamma^{-1}\gamma'$ does stabilize $Y^+ \cap Y^-$. Since it stabilizes $\partial\tilde{D_1}$  ($\partial\tilde{D_2}$) and $\tilde{f}$ is a $\pi_1$-equivariant map that sends $\partial\tilde{D_1}$  ($\partial\tilde{D_2}$) onto $Y^+ \cap Y^-$, we  conclude that $\tilde{f}'$ is so far well-defined.

Now $\tilde{f}'$ needs to be extended over the three orbifold balls that $\tilde{D_1}$ and $\tilde{D_2}$ divides $\tilde{B}$ into,  Call these $B_1$,$B_2$ and $B_3$.  To extend the map over $B_i$ start with any triangulation of $p(B_i)$ that lifts to a triangulation of $B_i$ (which is fine enough such that no $3$-cell has all of its vertices on the boundary of $p(B_i)$ and every $1$-cell and $2$-cell that has all of its vertices on the boundary are completely contained in the boundary).  Now follow the construction of a $\pi_1^{orb}$-equivariant map we used before sending all $0$-cells that are not in $\partial B_i$ to the vertex of $Y^+$ or $Y^-$ (note that for all $B_i$ the map $\tilde{f}$ sends $\partial B_i$ to the closure of either $Y^+$ or $Y^-$).  It should be noted that no points in the interior of $B_i$ are mapped to a point of $E$ by this construction.  This gives us a $\pi_1$-equivariant map $\tilde{f}':\tilde{Q}\to T$ that agrees with $\tilde{f}$ outside of $p^{-1}(B)$, and for which the compressing disk has been removed from $p(\tilde{f}'^{-1}(E))$.

Repeat the above steps until there are no more compressing orbifold disks for $F$.    Each time the process is performed the component that was cut either remains connected or is split into two new components of $F$.  When $F_0$ remains connected the genus of $F_0$ in the underlying space $X_Q$ is reduced.  This process cannot increase the genus of a component.  If a component is split into two, the sum of the genera of the two components is at most the genus of the original component.  Therefore, it is only possible   it is only possible to perform this process a finite number of times.  

When we perform this process and split a component $F_0$  into two components $F_0'$ and $F_1'$ by the compressing orbifold disc $D$, we do so because  the curve $\partial D$ in $F_0$ does not bound an orbifold disc in $F_0$.  As such,    the new components must have a positive genus, a non-empty boundary, or contain at least two singular points (at least three singular points if the disk $D$ has a singular point). Otherwise, one of the components of $F_0-\partial D$ is an open orbifold disk in $Q$ whose closure is a closed orbifold disk in $F$ bounded by $\partial D$.  Suppose that $F$ has a combined genus $g$, boundary components $b$ and singular points $s$ across all its components.  Then from the genus and boundary components the number of compressing disks it is possible to remove is bounded above by $2g-1+b$, each of which potentially adds a new singular point to $F$.  The maximum number of singular points $F$ can have is bounded above by $4g+2b$ when not yet counting those potentially added when creating a new component without boundary that has genus $0$.  When compressing along a disk that splits a component $F_0$ into two components $F_0'$ and $F_1'$ with at least one, say $F_0'$, having genus $0$ and no boundary, $F_1'$ will have less singular points than $F_0$ and $F_0'$ will have at most one more singular point than $F_0$.  However, both $F_0'$ and $F_1'$ will have less singular points than $F_0$ unless $F_1'$ has positive genus or a non-empty boundary.  The maximum number of compressing disks possible to remove without separating $F_0$ into two genus $0$ components with empty boundary is bounded above by $6g+3b+s$.  The maximum number of compressing disks removable from $F$ is bound above by $12g+3b+2s$.

We conclude that after a finite number of steps $F$ will no longer have any compressing orbifold disks and we can move on to removing boundary parallel and spherical suborbifold components from $F$. As a result,  we now handle the case when $F$ has boundary parallel or spherical suborbifold components. 

Suppose $F_0$ is an innermost boundary parallel component of $F$.  Then there exists a deformation retract $\rho$ of $Q$ onto a component of $Q-F_0$ that is constant and injective on neighborhoods of every other component of $F$.  Then there is a $\pi_1(Q)$-equivariant map $\tilde{\rho}:\tilde{Q}\to \tilde{Q}$ such that $p^{-1}(F_0)\cap\tilde{\rho}(\tilde{Q})=\emptyset$. Let $\tilde{f}'=\tilde{f}(\tilde{\rho})$.  Then $F_0\not\subset p(\tilde{f}'^{-1}(E))\subset F$.

Suppose $F_0$ is some spherical $2$-orbifold component of $F$.  Then $F_0$ is contained in the interior of some $3$-orbifold ball $B$ such that $\tilde{f}$ maps each component of $p^{-1}(\partial B)$ to a single point of $T$.  Define $\tilde{f}'$ to agree with $\tilde{f}$ for all $\tilde{Q}-p^{-1}(B)$ and let $\tilde{f}'$ map each component $\tilde{B}$ of $p^{-1}(B)$ to the point $\tilde{f}'(\partial\tilde{B})$.

The $2$-suborbifold $F$ in $Q$ defined by $\tilde{f}$ is incompressible, non-boundary parallel, and contains no essential orbifold $2$-spheres.  It remains to prove that $F$ is non-empty. From the construction of $F$ from the map $\tilde{f}$, $F$ can only be empty if $\tilde{f}$ maps all of  $\tilde{Q}$ into a component $V$ of $T-E$.  However, because $\pi_1^{orb}(Q)$ acts non-trivially on $T$ no component of $T-E$ is stabilized by $\pi_1^{orb}(Q)$.  Thus there exists some $\gamma\in\pi_1^{orb}(Q)$ such that $\gamma \cdot V=V'\neq V$.  Since $\tilde{f}$ is $\pi_1^{orb}$-equivariant, the image of $\tilde{f}$ must also contain a point in $V'$.  Because the image of $\tilde{f}$ is connected it contains a point in $E$.  Therefore, $F$ is a non-empty essential $2$-suborbifold. 
\end{proof}

\subsection{Lifting essential $2$-suborbifolds  }\label{section:lifting}

We will now show that when $Q$ is good the suborbifolds obtained from Theorem~\ref{orbifoldTreeAction}  lift to essential surfaces in the covering manifold.   
\begin{lemma}\label{suborbliftinglemma}
Let $F$ be an essential $2$-suborbifold in a good, orientable, orbifold-irreducible $3$-orbifold $Q$ such that $F$ contains only a finite number of singular points.  Let $p:M\to Q$ be an orbifold covering map by an irreducible manifold $M$.  Then $p^{-1}(F)$ is a symmetric essential surface in $M$. 
\end{lemma}

\begin{proof}
We will begin by proving that $F$ lifts to a bicollared surface in $M$. Since $Q$ is covered by a manifold $M$, the universal cover of $Q$ is also a manifold.  Away from the singular points on $F$, $F$ is bicollared surface.  So,  outside a neighborhood of each singular point on $F$, $F$ lifts to a bicollared surface in $M$.  Neighborhoods of the singular points can be made sufficiently small so that the orbifold structure in the neighborhood is that of a $2$-sphere modulo an action by a finite cyclic group.  

 Suppose some component $S$ of $p^{-1}(F)$ is a $2$-sphere.  Then $p(S)$ would be an essential spherical $2$-suborbifold in $Q$ because it is a $2$-sphere modulo the action of a finite group. This contradicts the fact that  $Q$ is   orbifold-irreducible.  Thus $p^{-1}(F)$ contains no $2$-sphere components.

 It follows   from \cite[Proposition $2.8$]{MR1926138} that every component of $p^{-1}(F_0)$ is incompressible for every component $F_0$ of $F$.  Thus, every component of $p^{-1}(F)$ is $\pi_1$-injective.


 Suppose a component $S$ of $p^{-1}(F)$ is boundary parallel.  Then there exists a homotopy $h$ of $S$ that takes $S$ into the boundary of $M$.  Then by the composition of continuous functions $p(h)$ is a homotopy that takes $S$ into the boundary of $Q$, contradicting the fact that  $F$ has no boundary parallel components.  Thus we conclude that $p^{-1}(F)$ contains no boundary parallel components, which completes the proof.

The surface is symmetric by construction since it is the pullback of a surface in $Q=M/G$. 
 \end{proof}

\section{Character Varieties }\label{section:charactervarieties}

\subsection{$\SL_2(\C)$ Character Varieties}\label{section:SL}

 We will use the standard terminology in the field, where a variety refers to an algebraic set which is not necessarily irreducible or smooth.
The {\em $\SL_2(\C)$ representation variety}  of the group $\Gamma$ is the  set 
\[ R(\Gamma)=\{\rho:\Gamma\to \SL_2(\mathbb{C})\}.\]
  Roughly speaking, these representations correspond to geometric structures on $M$, and conjugate representations correspond to isometric hyperbolic structures on $M$. As such it is natural to work with the character variety, which is $R(\Gamma)$ up to trace equivalence.  Formally, the {\em $\SL_2(\C)$ character variety} of $\Gamma$ is 
\[ X(\Gamma)  = \{ \chi_{\rho}: \rho\in R(\Gamma)\} \]
where the character function $\chi_{\rho}:\Gamma \rightarrow \C$ is defined as $\chi_{\rho}(\gamma) = \tr (\rho(\gamma))$ for all $\gamma \in \Gamma$. The natural surjection $t:R(\Gamma) \rightarrow X(\Gamma)$ which sends a representation to its trace is a regular map. 
Both $R(\Gamma)$ and $X(\Gamma)$ are complex algebraic sets defined over $\Q$. 
Isomorphic group yield isomorphic varieties, so we will often write $R(M)$ and $X(M)$ to denote $R(\pi_1(M))$ and $X(\pi_1(M))$ up to isomorphism.
A representation is called {\em reducible} if it is conjugate to an upper triangular representation, and {\em irreducible} otherwise. Two irreducible representations into $\SL_2(\mathbb{C})$ have the same trace function if and only if they are conjugate representations (see \cite{MR683804}).
A point $x\in X(\Gamma)$ equals $t(\rho)$ for some representation $\rho$ and we will often write it as $\chi_{\rho}$ with the understanding that $\rho$ is not uniquely determined.
 We will make use of the regular functions
\[ I_{\gamma}:X(\Gamma) \rightarrow \C\]
defined by $I_{\gamma}(\chi_{\rho})= \chi_{\rho}(\gamma)=\tr(\rho(\gamma))$ for a fixed $\gamma\in \Gamma$.
The $I_{\gamma}$ functions  extend to rational functions on a smooth projective completion of $X(M)$.

 In general, the  algebraic sets $R(M)$ and $X(M)$ are  not   irreducible  (as algebraic sets) and  have multiple components.   For example, the set of characters of reducible representations is easily seen to be an algebraic set in its own right.  
 We call a component a {\em canonical component} and write $R_0(M)$ (or $X_0(M)$)  if it contains a (character of a) discrete and faithful representation. 
We write $X(M)=X_{red}(M)\cup X_{irr}(M)$ where $X_{red}(M)$ contains the characters of irreducible representations, and $X_{irr}$ is the (affine) Zariski closure of $X(M)-X_{red}(M)$.  That is, $X_{irr}(M)$ is the Zariski closure of the union of components which each contain the character of irreducible representation.

We will call the character variety defined as above a {\em natural model} to distinguish it from a smooth projective completion.
A point in the Zariski closure of $X(M)$ is an {\em ideal point} if it does not correspond to a trace function. That is, the set of ideal points consists of those points in the projective closure of $X(M)$ that are not in $X(M)$.  
 
 We will call a  $\chi_{\rho} \in X(\Gamma)$ an {\em algebraic non-integral}  (or ANI) point if  $\rho(\Gamma)\subset \SL_2(\bar{\Q})$,  and there is a $\gamma \in \Gamma $ such that $\tr(\rho(\gamma)) $ is not contained in $\mathcal{O}_K$, where $\mathcal{O}_K$ is the ring of integers of a number field $K$. 
   Culler and Shalen  \cite{MR683804,MR881270}  showed the following   for an ideal   or  ANI point $x$  and Shanuel and Zhang \cite[Corollary 3]{MR1835066} generalized this to  other valuations.  Either 
\begin{enumerate}
\item There is a unique primitive element $\gamma$ such that $I_{\gamma}(x)\in \C$, \\
or 
\item $I_{\gamma}(x)\in \C$ for all $\gamma \in \pi_1(\partial M)$. 
\end{enumerate}
Strongly detected slopes correspond to the $\gamma$  to case $(1)$ and weakly detected slopes to case $(2)$.     In $(2)$, $\chi_{\rho}$ detects a closed essential surface.

\subsection{$\PSL_2(\C)$ Character Varieties}\label{section:PSL}

There are several different constructions for the $\PSL_2(\C)$-representation
and character varieties of $\Gamma$; we refer the reader to   \cite[Section 3]{MR1670053},
\cite[Section 2.1]{MR1739217},  and \cite{MR1248117}.

Let $\bar{R}(\Gamma)$ be the set of representations of $\Gamma$ into $\PSL_2(\C)$. This also has the structure of a complex algebraic set. The natural quotient map  \[ \Phi:\SL_2(\C) \rightarrow \PSL_2(\C)\] induces a regular map $\Phi_*:R(\Gamma) \rightarrow \bar{R}(\Gamma)$.  Each fiber is either empty or an orbit of the free $\Hom(\Gamma, \Z/2\Z)$ action. This is the sign change action on $R(\Gamma)$ defined by $(\epsilon \cdot \rho)(\gamma) = \epsilon(\gamma) \rho(\gamma)$ for $\epsilon \in \Hom(\Gamma, \Z/2\Z)$. A $\PSL_2(\C)$ representation $\bar{\rho}$ lifts to an $\SL_2(\C)$  representation $\rho$ exactly when the second Stiefel-Whitney class $\omega_2(\overline{\rho}) \in H^2(\Gamma; \Z/2\Z)$ vanishes.  The isomorphism class of $\omega_2(\bar{\rho})$ depends only on the component of $\bar{\rho}\in \bar{R}(\Gamma)$ \  and $\Phi_*$ defines a regular cover from $R(\Gamma)$ to its image. 

As in the $\SL_2(\C)$ case, the action of $\PSL_2(\C)$ on $\bar{R}(\Gamma)$ has a quotient $\bar{X}(\Gamma)$ with quotient map $\bar{t}:\bar{R}(\Gamma)\rightarrow \bar{X}(\Gamma)$ which is constant on conjugacy classes of representations.  The set $\bar{X}(\Gamma)$ is an affine algebraic set and is determined by the stipulation that its coordinate ring is isomorphic to the ring of invariants of the natural action of $\PSL_2(\C)$ on $\C[\bar{R}(\Gamma)]$.  For each $\gamma \in \Gamma$  the map $\bar{t}:\bar{X}(\Gamma) \rightarrow \C$ given by $\bar{t}(\bar{\rho}) = (\tr(\bar{\rho}(\gamma)))^2$ is regular,  and we write $\bar{t}(\bar{\rho})=\chi_{\bar{\rho}}$.  

We call $\bar{X}(\Gamma)$ the {\em $\PSL_2(\C)$ character variety of $\Gamma$}, and if $\Gamma$ is the fundamental group of a 3-manifold $M$ or orbifold $Q$ we use $\bar{X}(M)$ or $\bar{X}(Q)$ to denote $\bar{X}(\pi_1(M))$ and $\bar{X}(\pi_1^{orb}(Q))$ up to isomorphism. We call an irreducible component a {\em canonical component} if it contains $\bar{t}(\bar{\rho})$ where $\bar{\rho}$ is a discrete and faithful representation and write $\bar{X}_0$.

The function $\Phi$ induces a regular map $\Phi_{\#}:X(\Gamma) \rightarrow \bar{X}(\Gamma)$ with fibers that are either empty or the orbits of the $\Hom(\Gamma, \Z/2\Z)$ action  as $\epsilon \cdot \chi_{\rho}=\chi_{\epsilon\cdot \rho}$.  Indeed (see \cite{MR1670053} Lemma 3.1) two characters $\chi_{\bar{\rho}} $ and $\chi_{\bar{\rho'}}$ are equal exactly when the trace of $\rho(\gamma)$ equals the trace of $\rho'(\gamma)$ up to sign for all $\gamma \in \Gamma$.

Similar to the $I_{\gamma}$ functions defined above, for any $\gamma\in \Gamma$ we define $f_{\gamma}:\bar{X}(\Gamma) \rightarrow \C$   by 
\[ f_{\gamma}(\chi_{\bar{\rho}}) = \tr (\Phi^{-1}(\bar{\rho}(\gamma)))^2-4\] where the squaring is necessary because the sign is not well-defined.

We  call a point $\bar{\rho}\in \bar{R}(\Gamma)$ an {\em algebraic non-integral}  (or ANI) point if $\rho(\Gamma)\subset \PSL_2(\bar{\Q})$  and there is a $\gamma \in \Gamma $ such that $\tr(\Phi^{-1}(\bar{\rho}(\gamma)))^2$ is not contained in   $ \mathcal{O}_K$, where $\mathcal{O}_K$ is the ring of integers of a number field $K$.

\subsection{Induced Maps}\label{section:inducedmaps}

 A homomorphism  $\phi:A\rightarrow B$ between two groups $A$ and $B$ induces a map
\[  \begin{aligned}\phi^*: R(B) &  \rightarrow R(A) \\  \rho   & \mapsto \rho \circ \phi. \end{aligned} \] 
This descends to the $\SL_2(\C)$ character variety as
\[  \begin{aligned}\hat{\phi}: X(B) &  \rightarrow X(A) \\  \chi_{\rho}   & \mapsto \chi_{\phi^*(\rho)} \end{aligned} \] 
where $\chi_{\phi^*(\rho)} =\chi_{ \rho \circ \phi}.$
This is defined for all (affine) points on $R(B)$ or $X(B)$.  We obtain analogous maps in the $\PSL_2(\C)$ setting.  Re-using notation, we have the induced map
\[  \begin{aligned}\phi^*: \bar{R}(B) &  \rightarrow \bar{R}(A) \\  \bar{\rho}   & \mapsto \bar{\rho} \circ \phi. \end{aligned} 
\]
and for the $\PSL_2(\C)$ character variety we have
\[
 \begin{aligned}\hat{\phi}: \bar{X}(B) &  \rightarrow \bar{X}(A) \\  \chi_{\bar{\rho}}   & \mapsto \chi_{\bar{\rho}\circ \phi} \end{aligned} \] 
where $\chi_{\bar{\rho}}=\bar{t}(\bar{\rho})$, and $\chi_{\bar{\rho}\circ \phi}=\bar{t}(\bar{\rho}\circ \phi)$.

If $\hat{\phi}$ is dominant (its image is Zariski dense) then  $\hat{\phi}$ induces a $\C$-algebra homomorphism (see \cite[page 25]{MR0463157} ) between function fields $\C(X(A))\rightarrow \C(X(B))$.  For a regular function $f\in \C(X(A))$ its image is $f\circ \phi$ where it's understood that this is defined on an open subset of $X(A)$. If the mapping $\hat{\phi}$ is birational then it induces an isomorphism between  $\C(X(A))$ and $\C(X(B))$.
Considering the case when $X(A)$ and $X(B)$ (or $\bar{X}(A))$ and $\bar{X}(B)$) are curves,  it follows that $\hat{\phi}(X)$ is either Zariski dense in  $X(A)$  (or $\bar{X}(A)$) or it is a point.

Let $Q=M/G$ and consider the branched, or orbifold covering $p:M\rightarrow Q$ induced by the natural quotient.  This $p$ is an orbi-map (as defined in \cite{MR1065604} on pages 161-162) and therefore by \cite[Proposition 2.4]{MR1065604}   the induced homomorphism $p_*:\pi_1(M) \rightarrow \pi_1^{orb}(Q)$ is injective.   If $\tilde{\alpha}$ is a path in the underlying space of the universal cover of $M$, then $p_*([\tilde{\alpha}])= [\tilde{\alpha}]  $ where this second $[\tilde{\alpha}]$ is in the universal cover of $Q$.

\subsection{Symmetries and the Canonical Component}\label{section:symmandcanonical}

A symmetry $\sigma$ of a knot $K$ is an orientation preserving homeomorphism of $S^3$ that send $K$ to itself.  As a consequence of Mostow rigidity the full group of symmetries of $K$ is finite and acts on $M=S^3-K$ by isometries (see \cite{thurston} Corollary 5.7.4). The symmetry group of a knot in $S^3$ is either cyclic or dihedral \cite{MR1646740}.    By the Gordon-Luecke theorem \cite{MR965210} any homeomorphism of knot complements must take meridian to meridian, and therefore take an (un-oriented) longitude to itself as well.   In general, if $M$ is a one cusped hyperbolic 3-manifold  we define a {\em symmetry} of $M$ to be any orientation-preserving isometry of $M$.  We will only consider symmetries that preserve the framing of the the cusp.

If $\sigma:M\rightarrow M$ is a symmetry of $M$ then $\sigma$ induces an automorphism of  $\sigma_*$ of $\pi_1(M)$. This induces an automorphism  $\sigma^*$ of the representation variety $R(M)$ by $\sigma^*(\rho)=\rho\circ \sigma_*$, and an automorphism  $\hat{\sigma}$ of the $\SL_2(\C)$ character variety $X(M)$ by $\hat{\sigma}(\chi_{\rho})=\chi_{\sigma^*(\rho)}$.  
An analogous statement can be made in the $\PSL_2(\C)$ setting.
These are all regular maps.

In this section, we prove the following theorem.
\begin{thm}\label{thm:symmetiesandcanonicalcomponent}
Let $\sigma$ be a  symmetry of the one cusped hyperbolic manifold $M$ with induced symmetry $\hat{\sigma}$ on  $\bar{X}(M)$.  A canonical component $\bar{X}_0(M)$  of the $\PSL_2(\C)$ character variety  is a subset of the fixed point set of $\hat{\sigma}$.  For $g\in \pi_1(M)$, and a point $\chi_{\rho}$ on a canonical component of the $\SL_2(\C)$ character variety, $X_0(M)$, we have   $\hat{\sigma}(\chi_{\rho}(g)) = \pm \chi_{\rho}(g)$.
 If for all $g\in \pi_1(M)$,  $\sigma_*(g)$ and $g$ have the same image in the surjection $\pi_1(M) \rightarrow \Z/2\Z$ then $X_0(M)$ is  a subset of the fixed point set of $\hat{\sigma}$. 
\end{thm}

\begin{proof}
Since $M$ has only one cusp, $\sigma$ induces a symmetry on $\partial M$ which retains the framing of $\partial M$.    Let $M(p/q) = M\cup_{\partial M} T$ be $p/q$ Dehn filling of $M$.  It follows that $\sigma $ extends to a symmetry of $M(p/q)$.

By Thurston's hyperbolic Dehn surgery theorem, all but finitely many Dehn surgeries on $M$ are hyperbolic. For such a $M(p/q)$,   by Mostow rigidity there is a unique complete hyperbolic structure on $M(p/q)$ therefore   representations of $\pi_1(M(p/q))$ correspond to a holonomy representation to $\PSL_2(\C)$.  Any two such representations differ only by conjugation and perhaps complex conjugation. Complex conjugation corresponds to an orientation reversal. 

By the Seifert-Van Kampen theorem, $\pi_1(M(p/q))$ is isomorphic a quotient of $\pi_1(M)$, so there is a surjection $p:\pi_1(M) \rightarrow \pi_1(M(p/q))$.   Let $\rho:\pi_1(M)\rightarrow \mathrm{(P)SL}_2(\C)$ be a representation.  Then we obtain a representation $\rho \circ p: \pi_1(M)\rightarrow \mathrm{(P)SL}_2(\C)$.  It follows that $R(M(p/q))$ naturally embeds in $R(M)$, and similarly $X(M(p/q))\subset X(M)$ and $\bar{X}(M(p/q))\subset \bar{X}(M)$. 
Thurston's theorem shows that there are faithful representations of $\pi_1(M(p/q))$ for all but finitely many $p/q$  in any Riemannian neighborhood of any discrete and faithful representation of $\pi_1(M)$.  Therefore,   there are infinitely many of these points on any $X_0(M)$ or $\bar{X}_0(M)$ and they form a Zariski dense set. 
We conclude that  it is enough to show that the Dehn surgery points on $X_0(M)$ (or $\bar{X}_0(M)$) are fixed by $\hat{\sigma}$.

Assume that $M(p/q)$ is hyperbolic and consider a loop $\gamma\in M(p/q)$.  As $\sigma$ extends to $M(p/q)$, the length of a geodesic representative for $\gamma$, and the length of a geodesic representative for $\sigma(\gamma)$ must be equal.  
Since $M(p/q)$ is a closed hyperbolic 3-manifold all loops correspond to loxodromic elements in the holonomy representation where $\gamma$ corresponds to $\pm A \in \PSL_2(\C)$.  Let $\ell_0(\gamma)$ denote the  (complex) translation length of  a geodesic representative for $\gamma$. Then (see \cite{MR1937957} p 372) 
\[
\pm \tr \gamma/2 = \cosh(\ell_0(\gamma)/2). 
\]
If a holonomy representation takes  $\sigma(\gamma)$ to a matrix $\pm B$ we see that $ \pm \tr A = \tr B.$  
Therefore, for all $\rho:\pi_1(M)\rightarrow \PSL_2(\C)$ that correspond to holomony representations of hyperbolic Dehn fillings (since $\sigma$ is orientation preserving) and all $[\gamma] \in \pi_1(M(p/q))$ we have $\chi_{\rho}([\gamma]) = \chi_{\rho}(\sigma_*([\gamma])$. In other words, $\hat{\sigma}(\chi_{\rho})=\chi_{\rho}$.  Therefore, $\hat{\sigma}$ is fixed in $\bar{X}_0(M)$.

It follows that $\hat{\sigma}:X_0(M) \rightarrow X_0(M)$ satisfies  $\hat{\sigma}(\chi_{\rho}(g)) = \chi_{\rho}(\sigma_*(g)) = \pm \chi_{\rho}(g)$ for all $g\in \pi_1(M)$. By the discussion in Section~\ref{section:PSL} we conclude that the signs will be equal if $g$ and $\sigma(g)$ are in the same coset.  By definition, this occurs when $g$ and $\sigma_*(g)$ have the same image in the map to $\Z/2\Z$.   
\end{proof}

If $\sigma$ is a symmetry that fixes unoriented free homotopy classes of loops, then the induced action $\hat{\sigma}$ on $\bar{X}(M)$ is trivial;  the induced action on the representation variety sends an element of $\mathrm{(P)SL}_2(\C)$ to a conjugate or the inverse of a conjugate.  Indeed there are examples \cite{MR3078072} where the induced action $\hat{\sigma}$ is trivial even when $\sigma$ does not fix unoriented free homotopy classes of loops.

When $M$ is a knot complement in $S^3$ then the condition that for all $g\in \pi_1(M)$,  $\sigma_*(g)$ and $g$ have the same image in the surjection $\pi_1(M) \rightarrow \Z/2\Z$ holds, so we have the following immediate corollary.
\begin{cor}\label{cor:symmetiesandcanonicalcomponent}
 Let $K$ be a hyperbolic knot in $S^3$ with exterior $M$.  A canonical component   of the $\PSL_2(\C)$ character variety, $\bar{X}_0(M)$ is a subset of the fixed point set of $\hat{\sigma}$. A  canonical component of the $\SL_2(\C)$ character variety, $X_0(M)$, is  a subset of the fixed point set of $\hat{\sigma}$. 
\end{cor}

\section{ $\mathrm{(P)SL}_2(\C)$ Culler-Shalen Theory for Orbifolds }\label{section:PSLOrbs}

We have proven Theorem~\ref{orbifoldTreeAction}, which says that  given an action of $\pi_1^{orb}(Q)$ on a tree that is non-trivial and without inversions, we can detect an essential 2-suborbifold that is dual to the action. Now we show how to use valuations from $\SL_2(\C)$ and $\PSL_2(\C)$ character varieties to obtain such actions. Specifically, we show that the valuation theory aspect of Culler and Shalen's work holds for orbifolds in both the $\SL_2(\C)$ and $\PSL_2(\C)$ settings.

Let $F$ be a field with discrete valuation $v:F^*\rightarrow \Z$.  The valuation ring is 
\[ \Ox_v=\{0\} \cup \{ a\in F^*: v(a)\geq 0\}. \]We construct a simplicial tree $T=T_{F,v}$ on which $\SL_2(F)$ and $\PSL_2(F)$ act simplicially and without inversions. 

The tree $T$ is defined as follows.  The vertices of $T$ are the homothety classes of lattices in $F^2$ and for lattices $\Lambda_1$ and $\Lambda_2$ the distance between the vertices $[\Lambda_1]$ and $[\Lambda_2]$ is $v(\det A)$ where $A:F^2\rightarrow F^2$ is a linear transformation taking $\Lambda_1$ to $\Lambda_2$.  The group $\SL_2(F)$ acts on $T$ by linear automorphisms.  
By Theorem~\ref{orbifoldTreeAction},   an action that is non-trivial and without inversions (non-canonically) identifies essential $2$-suborbifolds dual to it. 
Note that the action of $-I\in \SL_2(F)$ fixes every vertex of $T$ as it takes any lattice $\Lambda$ to itself.  Let $\Phi:\SL_2(F) \rightarrow \PSL_2(F)$ be the natural map.  
It follows that there is a well-defined action of $\PSL_2(F)$ on $T$ as follows.  Let $\bar{A} \in \PSL_2(F)$; we define $\bar{A}\cdot T= A\cdot T$ where $A\in \Phi^{-1}(A)$.  (Given $\psi:\Gamma \rightarrow \PSL_2(F)$, we are not assuming that $\psi(\Gamma)\subset \PSL_2(F)$ lifts to $\SL_2(F)$; we are simply using the pull-back.)

The action of an element $A \in \SL_2(F)$ on $T$ fixes a vertex if and only if $A$ is conjugate in $\mathrm{GL}_2(F)$ to an element in $\SL_2(\Ox_v)$.  This is equivalent to   $\tr(A) \in  \Ox_v$.  Therefore if there is an $A \in \SL_2(F)$ such that $v(A)<0$ this $A$ cannot stabilize any vertex and the action is non-trivial.

Let ${\rm GL}_2(F)^+$ be the kernel of the map ${\rm{GL}}_2(F)\rightarrow \Z \rightarrow \Z/2\Z$. 
By \cite[Section II 1.2-1.3]{MR0476875}, the group ${\rm GL}_2(F)$ acts with inversions, but  ${\rm GL}_2(F)^+$ acts on $T$ without inversions.  Since, $\SL_2(F)\subset {\rm GL}_2(F)^+,$   both $\SL_2(F)$ and $\PSL_2(F)$ act without inversions.

With Theorem~\ref{orbifoldTreeAction} we have the following, as in \cite[Lemma 1]{MR1835066}.  
\begin{prop}\label{prop:valuationgivesessential}
Let $v$ be a discrete valuation on a field $F$ and let $T=T_{F,v}$ be the associated  tree.  If a  finite volume 3-orbifold $Q$  has a representation $\psi:\pi_1^{orb}(Q)\rightarrow \mathrm{(P)SL}_2(F)$ with $v(\tr(\psi(\gamma)))<0$ for some $\gamma \in \pi_1^{orb}(Q)$ then this induces a non-trivial action of $\pi_1^{orb}(Q)$ on $T$ and dual to this action is an essential $2$-suborbifold of $Q$.
\end{prop}

Similarly, the following is essentially \cite[Corollary 3]{MR1835066} using  \cite[Lemma 2]{MR1835066} and follows from Lemma~\ref{lemma:stabilizers}.
\begin{prop}\label{prop:generaltype12}
Let $Q$ be finite volume   3-orbifold. 
Let $v$ be a discrete valuation on a field $F$,  let $T=T_{F,v}$ be the associated  tree and let $\psi:\pi_1^{orb}(Q)\rightarrow \mathrm{(P)SL}_2(F)$ such that there is a $\gamma \in \pi_1^{orb}(Q)$ with $v(\tr(\psi(\gamma)))<0$. Then
\begin{enumerate}
\item There is a unique primitive element $\gamma \in \pi_1^{orb}(\partial Q)$ such that $v(\tr(\psi(\gamma))) \geq 0$.  Then $Q$ contains an essential $2$-suborbifold dual to the action with peripheral element $\gamma$.  
\item For all $\gamma \in \pi_1^{orb}(\partial Q)$, we have $v(\tr(\psi(\gamma)))\geq 0$.  Then $Q$ contains a closed essential $2$-suborbifold dual to the action.
\end{enumerate}
\end{prop}
Note that $(1)$ corresponds to having a unique primitive element $\gamma$ with $\psi(\gamma) \in \Ox_F$, and $(2)$ corresponds to all $\psi(\gamma) \in \Ox_F$. Now we discuss valuations from ANI points and points at infinity in a bit more detail.

\subsection{ANI-Point}

See \cite[page 51]{MR1835066}  for a discussion of this in the $\SL_2(\C)$ case for manifolds.

Let $\Gamma$ be a finitely generated group,  and assume that there is a representation $\rho:\Gamma \rightarrow \SL_2(F)$, or $\bar{\rho}:\Gamma  \rightarrow \PSL_2(F)$ where $F$ is a number field.  An {\em algebraic non-integral or ANI point} is  a  $\chi_{\rho}$ or $\chi_{\bar{\rho}}$ where there is $\gamma \in \Gamma$ such that $\tr(\rho(\gamma))\not \in \Ox_F$, or $\pm \tr(\bar{\rho}(\gamma)) \not \in \Ox_F$. This   is equivalent to the definition given earlier in Sections~\ref{section:SL} and ~\ref{section:PSL}. 
Therefore, there is some prime ideal $\pi$ of $\Ox_F$ such that the $\pi$-adic valuation $v$ has $v(  \tr(\rho(\gamma)))<0$, or similarly $v(\pm \tr (\bar{\rho}(\gamma)))<0$.  Since $v(-1)=1$   this is well-defined.  
It immediately follows that this action is non-trivial and without inversions.

 In this case, we can rephrase Proposition~\ref{prop:generaltype12} as the following.  Given an ANI point $x$ in $X(Q)$ (or $\bar{X}(Q)$) then either
\begin{enumerate}
\item There is a unique primitive element $\gamma\in \pi_1^{orb}(\partial Q)$ such that $I_{\gamma}(x)$ (or $f_{\gamma}(x)$) is in $\C$.  Then $Q$ contains an essential 2-suborbifold dual to the action with peripheral element $\gamma$. \\
or 
\item For all $\gamma \in \pi_1^{orb}(\partial Q)$ we have $I_{\gamma}(x)$  (or $f_{\gamma}(x)$) in $\C$. Then $Q$ contains a closed  essential 2-suborbifold dual to the action. 
\end{enumerate}
This is analogous to the statement for manifolds.

\subsection{Ideal Point}
 
Let $\Gamma$ be a finitely generated group.    Let $X$ be a curve component of $X(\Gamma)$ and identify $X$ with a smooth projective model.  Then $I_{\gamma}$ is an element of the function field of $X$, and we can consider $I_{\gamma}$ to be a meromorphic function.  For a fixed $x\in X$ there is a valuation $v_x$ on the function field $\C(X)$ which assigns to each rational function its order at $x$.  This valuation extends to a valuation $w$ on $K$, the function field of the subvariety $R$  of $R(M)$ that maps to $X$ under the trace map.  Using $w$,  $\SL_2(K)$ acts on $T=T_w$ by restricting the action of $\GL_2(K)$.

 Using the  tautological representation $\mathcal{P}:\Gamma \rightarrow \SL_2(K)$,  the action of $\SL_2(K)$ on $T$ can be pulled back via $\mathcal{P}$ to an action of $\Gamma$ on $T$.  This action is without inversions, and is non-trivial if $x$ is an ideal point.  
This works as follows.  The tautological representation  is $\mathcal{P}:\Gamma \rightarrow \SL_2(\C[X])$ where $\C[X]$ is the coordinate ring of $X$, is defined by 
\[
\mathcal{P}(\gamma) = \mat{a_{\gamma}}{b_{\gamma}}{c_{\gamma}}{d_{\gamma}}
\]
where the coordinates are functions of the representations.  For example, $a_{\gamma}$ is the element of the coordinate ring $\C[X]$ that for an $x \in X$ produces $a_{\gamma}(x)\in \C$. 
This gives an action of $\Gamma$ on $T_w$ by 
\[ \gamma \cdot [\Lambda] = [\mathcal{P}(\gamma) \cdot \Lambda ] \]
specifically, if $\{e,f\}$ is a basis for $\Lambda$ we have $ \gamma \cdot [\Lambda] $ given by 
\[ \mat{a_{\gamma}}{b_{\gamma}}{c_{\gamma}}{d_{\gamma}} \left(  \begin{array}{c} e \\ f \end{array} \right) = \left(  \begin{array}{c}   a_{\gamma} e +b_{\gamma}f \\ c_{\gamma}e + d_{\gamma} f\end{array} \right)
\]
which determines a basis for a lattice.  (The result is the homothety class of this lattice.) Note that $a_{\gamma}, b_{\gamma}, c_{\gamma}, d_{\gamma}, e,$ and $f$ are all in $K$.

Boyer and Zhang \cite[Section 4]{MR1670053} extended this theory to $\PSL_2(\C)$ character varieties. Using a central $\Z/2\Z$ extension, they define a compatible tautological representation $\bar{\mathcal{P}}: \Gamma \rightarrow \PSL_2(\C)$.  If there is a curve component $X$ in $X(\Gamma)$ that under the natural map, maps to $\bar{X}$ in $\bar{X}(\Gamma)$, then  we see that for an order of vanishing valuation (at an ideal point), $I_{\gamma}$ blows up at an ideal point $x\in X$ if and only if $f_{\gamma}$ blows up at the image of $x$ in $\bar{X}$.  

The construction is as follows.  
 They show (Section 3) that for a finitely generated group $\Gamma$ and representation $\bar{\rho}:\Gamma \rightarrow \PSL_2(\C)$  there is a central extension $\phi:\hat{\Gamma}\rightarrow \Gamma$ of $\Gamma$ by $\Z/2\Z$ and a representation $\hat{\rho}:\hat{\Gamma}\rightarrow \SL_2(\C)$ such that the commutative diagram 
\[
\begin{tikzcd}
1 \ar[r] & \Z/2\Z \ar[d,equal] \ar[r] & \hat{\Gamma} \ar[d, "\hat{\rho}"]  \ar[r,"\phi"]& \Gamma \ar[d,"\bar{\rho}"] \ar[r] & 1 \\
1 \ar[r]  & \Z/2\Z \ar[r]  & \SL_2(\C) \ar[r, "\Phi"]  & \PSL_2(\C) \ar[r]  & 1 
\end{tikzcd}
\]
is exact. We call $(\hat{\Gamma}, \phi, \hat{\rho})$ the central $\Z/2\Z$ extension of $\Gamma$ lifting $\bar{\rho}$.  As mentioned in \cite{MR1670053} the central extensions of $\Gamma$ by $\Z/2\Z$ are classified by $H^2(\Gamma, \Z/2\Z)$  and $\bar{\rho}$ determines an element $w_2(\bar{\rho}) \in H^2(\Gamma, \Z/2\Z)$.  This is zero if and only if $\hat{\Gamma} \cong \Gamma \times \Z/2\Z$.  (That is, if $\hat{\rho}$ lifts to an element of $R(\Gamma)$.)  When $H^2(\Gamma, \Z/2\Z)$ is trivial, $\hat{\Gamma}\cong \Gamma \times \Z/2\Z$ for every $\bar{\rho}$ and each $\bar{\rho}$ lifts to $R(\Gamma)$.

Let $\Phi:\SL_2(\C)\rightarrow \PSL_2(\C)$ be the natural quotient. Let $\bar{X} \subset \bar{X}(\Gamma)$ be a curve. Then there is a subvariety $\bar{R}\subset \bar{t}^{-1}(\bar{X})$ which is 4-dimensional. Let $\bar{\rho}$ be a smooth point in $\bar{R}$ and $(\hat{\Gamma}, \phi, \hat{\rho})$ the central extension. Then there is a variety $S\subset R(\hat{\Gamma})$ containing $\hat{\rho}$ and a dominating regular map $\phi_*:S\rightarrow \bar{R}$. The function field $\C(S)$ is a finitely generated extension of $\C(\bar{X})$ by identifying $f\in \C(\bar{X})$ with $f\circ \bar{t} \circ \phi_* \in \C(S)$.

There is a tautological representation $\hat{\mathcal{P}}:\hat{\Gamma}\rightarrow \SL_2( \C(S))$ as discussed above.   The tautological representation $\bar{\mathcal{P}}:\Gamma \rightarrow \PSL_2(\C(S))$ is given by 
\[ 
\bar{\mathcal{P}}(\gamma)   = \Phi( \hat{\mathcal{P}}(\phi^{-1}(\gamma))).
\]
So that 
\[
\bar{\mathcal{P}}(\gamma)(\hat{{\rho}}) = \bar{\rho}(\gamma)
\]
for every $\gamma\in \Gamma$ and $\bar{\rho}\in \bar{R}$ and $\hat{\rho}\in S$ such that $\phi_*( \hat{\rho})=\bar{\rho}$.

We can rephrase Proposition~\ref{prop:generaltype12} as the following.  Given an ideal point $x$ in $X(Q)$ (or $\bar{X}(Q)$) then either
\begin{enumerate}
\item There is a unique primitive element $\gamma\in \pi_1^{orb}(\partial Q)$ such that $I_{\gamma}(x)$ (or $f_{\gamma}(x)$) is in $\C$.  Then $Q$ contains an essential 2-suborbifold dual to the action with peripheral element $\gamma$. \\
or 
\item For all $\gamma \in \pi_1^{orb}(\partial Q)$ we have $I_{\gamma}(x)$  (or $f_{\gamma}(x)$) in $\C$. Then $Q$ contains a closed  essential 2-suborbifold dual to the action. 
\end{enumerate}
This is analogous to \cite[Proposition 4.7]{MR1670053}

\section{The proof of Theorem~\ref{thm:maintheorem}} \label{mainsection}

If $F$ is detected by an action of $\pi_1^{orb}(Q)$ on a tree $T$ we now show that there is an  action of  $\pi_1(M)$ on $T$ which detects a symmetric essential surface in $M$.  
 Let $p_Q:\tilde{Q}\rightarrow Q$, $p_M:\tilde{Q}\rightarrow M$, and $p:M\rightarrow Q$ be the covering maps so that  $p\circ p_M=p_Q$. 
The action of $\pi_1^{orb}(Q)$ on $T$ gives a map $\tilde{f}:\tilde{Q}\rightarrow T$, which induces the maps $f_Q= \tilde{f} \circ p_Q^{-1} $ and $f_M= \tilde{f} \circ p_M^{-1}$ where $f_Q:Q\rightarrow T$ and $f_M:M\rightarrow T$.  
\[
\begin{tikzcd}
\tilde{Q} \ar[dr,"\tilde{f}"]  \ar[d,"p_M"]   \arrow[dd, bend right=60, swap,  "p_Q"]   & {} \\
M \ar[r,"f_M"]  \ar[d,"p"]  & T   \\
Q \ar[ur, swap,"f_Q"] & {}
\end{tikzcd}
\]
The map $f_Q$ is  well-defined because of the $\pi_1^{orb}(Q)$ equivariance of $\tilde{f}$.  The same is true of $f_M$ because $\pi_1(M)$ injects into $\pi_1^{orb}(Q)$.

\begin{lemma}\label{lemma:liftedsurfaceisdetected} 

If $F_Q$ is detected by an action of $\pi_1^{orb}(Q)$ on a tree $T$ then $p^{-1}(F_Q)$ is a symmetric essential surface in $M$ detected by the induced action of $\pi_1(M)$ on $T$.  
\end{lemma}

\begin{proof}

By Section~\ref{actionontreesection}, the essential 2-suborbifold $F_Q$ of $Q$ is associated to a $\pi_1^{orb}(Q)$ equivariant map $\tilde{f}:\tilde{Q}\rightarrow T$ which is transverse of $E$. The map $\tilde{f}$ induces the maps $f_Q$ and $f_M$ as above.  

By Theorem~\ref{orbifoldTreeAction},    $\tilde{F}=\tilde{f}^{-1}(E)$ is an essential surface in $\tilde{Q}$ and $F_Q=p_Q(\tilde{F})=f_Q^{-1}(E).$
 Now, $\tilde{f}$ is $\pi_1^{orb}(Q)$ equivariant, and since $\pi_1(M)$ injects into $\pi_1^{orb}(Q)$ it is also $\pi_1(M)$ equivariant.     By Lemma~\ref{suborbliftinglemma},  $F_M=p_M(\tilde{F})=f_M^{-1}(E)$ is a symmetric  essential surface in $M$  because $F_M=p^{-1}(F_Q)$.   
Because $F_M=f_M^{-1}(E)$ it is the surface detected by the induced action of $\pi_1(M)$ on $T$. 
\end{proof}

\begin{lemma}\label{lemma:SL=PSL}
Assume that $\bar{X}(M)$ lifts to $X(M)$, and let  $x\in X(M)$ or   an ideal point of $X(M)$. Let $\bar{x}$ be the image of $x$ under the covering map $X(M)\rightarrow \bar{X}(M)$.  Assume that  associated to $x$ there is 
a discrete valuation $v$  on a field $F$ where $T$ is the associated $\SL_2$-tree, and there is a representation $\phi:\pi_1(M) \rightarrow \SL_2(F)$ with $v(\tr(\phi(\gamma)))<0$ for some $\gamma \in \pi_1(M)$.  Then associated to $\bar{x}$ there is an action on  $T$ and if $x$ detects a surface $F$ then $\bar{X}$ also detects $F$.

  Similarly, assume that  associated to $\bar{x}$ there is 
a discrete valuation $v$  on a field $F$ where $T$ is the associated  tree, and there is a representation $\phi:\pi_1(M) \rightarrow \PSL_2(F)$ with $v(\tr(\phi(\gamma)))<0$ for some $\gamma \in \pi_1(M)$. 
  Then   $x$ has an action on $T$  and if $\bar{x}$ detects $F$ then $x$ also detects $F$.

\end{lemma}

\begin{proof}
The points $x$ and $\bar{x}$ are either both affine points or both ideal points. 
If they are affine points, the  representations associated to  $x$ and $\bar{x}$ are different only by signs, so any valuation associated to one point induces an identical valuation on the other.  If they are ideal points, the order of vanishing valuations are exactly the same. Since $-I$ acts on the tree by fixing every vertex, the lemma follows. 
\end{proof}

Now we are ready to prove Theorem~\ref{thm:maintheorem}, which we  restate:
\maintheorem*

\begin{proof} 
 
By Lemma~\ref{lemma:SL=PSL} it is enough to consider the $\PSL_2(\C)$ character variety.
By Section~\ref{actionontreesection} there is a $\pi_1(M)$ equivariant map $\tilde{f}:\tilde{Q} \rightarrow T$ that is transverse to $E$ such that the surfaces $\tilde{F}=\tilde{f}^{-1}(E)$ and $F_M= p_M(\tilde{F})=f_M^{-1}(E)$ are two sided and bicollared and $F_M$ is essential.  

Recall that $p:M\rightarrow Q$ is the covering map which induces $p_*:\pi_1(M) \rightarrow \pi_1^{orb}(Q)$.  The map  $\hat{p}:\bar{X}(Q) \rightarrow \bar{X}(M)$ is given by the following:   if $\chi_{\bar{\rho}} \in \bar{X}(Q)$ then $\hat{p}(\chi_{\bar{\rho}}) = \chi_{\bar{\rho}'}$ where for $\gamma\in \pi_1(M)$ we have 
\[ \chi_{\bar{\rho}'}(\gamma) = \chi_{\bar{\rho}\circ p_*}(\gamma) = \chi_{\bar{\rho}}(p_*(\gamma)). \]
It follows that $\hat{p}$ is defined for all affine points of $\bar{X}(Q)$, and
if $\chi_{\bar{\rho}}$ is an affine point, then so is its image, $\chi_{\bar{\rho}'}$.

Long and Reid \cite[Theorem 3.4]{MR1739217} proved that given a covering $p:M\to Q$ as above where $Q$ has a flexible cusp, then 
the induced map  $\hat{p}$   is a birational equivalence.
  Let $X$ be a smooth projective completion of $\bar{X}(Q)$.  Then $X$ is also a smooth projective completion of $\bar{X}(M)$, and $\hat{p}$ induces an automorphism of $X$. 
If $\bar{x}$ is an ideal point of $\bar{X}(Q)$  then there is a $\gamma \in \pi_1^{orb}(Q)$ such that $f_{\gamma}(\bar{x})$ blows up.  Assume that $\hat{p}(\bar{x})=\bar{x}'$.    Then the function $f_{p_*(\gamma)}(\bar{x}')$ is the image of $f_{\gamma}(\bar{x})$ under the induced mapping.  The map $p_*$ is an injection, and since the cover is finite, $p_*(\pi_1(M))$ is a finite index subgroup of $\pi_1^{orb}(Q)$.  But since $f_{\gamma}(\bar{x})$ blows up, so does $f_{\gamma^n}(\bar{x})$ for all  $n$.  In particular, $\gamma$ cannot be torsion.  Moreover, there are affine points $\chi_{\bar{\rho}_i}$ converging to $\bar{x}$ where $f_{\gamma}(\chi_{\bar{\rho}_i})$ is unbounded.  The images, under the map induced by $p$, of infinitely many $\chi_{\bar{\rho}_i}$ are in any neighborhood of the image of $\bar{x}$.   Taking $n$ large enough so that $  \gamma^n\in p_*(\pi_1(M))$ we let $\delta\in \pi_1(M)$ such that $p_*(\delta)=\gamma^n$. 
The corresponding evaluation functions are given by  
\[  f_{\delta}(\chi_{\bar{\rho}_i\circ p_*}) = \tr(\bar{\rho}_i(p_*(\delta)))^2-4 \] 
and we see that these evaluation functions are unbounded. We conclude that if $\bar{x}$ is an ideal point of $\bar{X}(Q)$ then its image is an ideal point of $\bar{X}(M)$. 
We have shown that $\hat{p}$ is defined on all of $\bar{X}(Q)$, and affine points map to affine points and (in the natural extension of $\hat{p}$ to smooth projective models) ideal points map to ideal points.

First, we consider the case when $\phi$ is associated to an ideal point $\bar{x}$ and the corresponding valuation is the order of vanishing valuation. Below we will consider representations as extended to  $\bar{X}$, a smooth projective completion of $\bar{X}_0(M)$. Let $\Gamma=\pi_1(M)$ and let $\Gamma'=\pi_1^{orb}(Q)$.  We refer the reader to  \cite[Section 4]{MR1670053} for details about the following construction.    The ideal point $\bar{x}$ detects a surface from the action of $\pi_1(M)$ on the tree $T$.  Let $\bar{R}\subset \bar{t}^{-1}(\bar{X})$ so that $\bar{R}\subset \bar{R}(\Gamma)$.  
For $\bar{\rho} \in \bar{R}$ such that $\bar{t}(\rho)=\bar{x}$, let  $(\hat{\Gamma}, \phi, \hat{\rho})$ be the $\Z/2\Z$ central extension of $\Gamma$ lifting $\bar{\rho}$.  Then there is a $S\subset R(\hat{\Gamma})$ containing $\hat{\rho}$ and a dominating regular map $\phi_*:S\rightarrow \bar{R}$. The function field $\C(S)$ is a finitely generated extension of the function field $\C(\bar{X})$. An analogous construction follows for $\Gamma'$ where we will embellish the relevant spaces with a $'$ symbol. Here the relevant central extension will be associated to the representation $p_*(\bar{\rho}):\Gamma'\rightarrow \PSL_2(\C)$.

The action corresponding to  $\bar{x}$  is induced by the tautological representation $\bar{\mathcal{P}}:\pi_1(M)\rightarrow \PSL_2(\C(S))$.  
Let $\bar{x}'\in \bar{X}(Q)$ such that $\hat{p}(\bar{x}')=\bar{x}$.  Such an $\bar{x}'$ exists because the map extends to an automorphism $\bar{X}\rightarrow \bar{X}$ since $\bar{X}$ is a smooth projective completion of both $\bar{X}(Q)$ and $\bar{X}(M)$.   The function field of $\bar{x}'$ and the function field of $\bar{x}$ are therefore isomorphic and are contained in $\C(S)$.  

Both  $\bar{x}$ and $\bar{x}'$ with the order of vanishing valuation define actions on the tree $T$  for $\SL_2(\C(S))$.  The fact that $\bar{x}$ has a negative valuation means that in the unique extension of the valuation of $\C(S)$ it also values negatively.  As stated above, since $\bar{x}$ is an ideal point, so is $\bar{x}'$, and so it has a negative valuation that extends to $\C(S)$ as well.   By Theorem~\ref{orbifoldTreeAction} there is a surface $F_Q$ in $Q$ dual to the action associated to $\bar{x}'$.  
By Lemma~\ref{lemma:liftedsurfaceisdetected} the surface, $F_Q$    lifts to a symmetric surface $F_M$ in $M$ which is detected by the induced action of $\Gamma=\pi_1(M)$ on $T$.  It remains to show that this induced action is an action at $\bar{x}$.  That is, the induced action from Lemma~\ref{lemma:liftedsurfaceisdetected} is the restriction $\mathcal{P}'\hspace{-0.1cm}\mid_{\Gamma}$ of the tautological representation $\mathcal{P}':\Gamma'\rightarrow \PSL_2(\C(S))$. It is left to show that this is the tautological representation $\mathcal{P}:\Gamma\rightarrow \PSL_2(\C(S))$ associated to $\bar{x}$. If $\bar{x}$ is associated to the representation $\bar{\rho}:\Gamma\rightarrow \PSL_2(\C)$, then the tautological representation satisfies $\mathcal{P}(\gamma)(\hat{\rho})=\bar{\rho}(\gamma)$ for all $\gamma\in \Gamma$. 
The restriction of $\mathcal{P}'$ to $\Gamma$ satisfies $\mathcal{P}'(p^*(\gamma))(\hat{\rho}')=\bar{\rho}'(p^*(\gamma))$, so it is enough to see that 
\[ 
\bar{\rho}'(p^*(\gamma)) = \bar{\rho}(\gamma)
\]
which is the case since $\hat{p}(\bar{x}')=\bar{x}$ so that $p_*(\bar{\rho}')=\bar{\rho}$ and therefore $p_*(\bar{\rho}')(\gamma) = \bar{\rho}(p^*(\gamma))$ as needed.

Finally, we consider the case when $\phi$ is associated to an affine point $\chi_{\bar{\rho}} \in \bar{X}(M)$.  Then $\bar{\rho}:\pi_1(M) \rightarrow \PSL_2(F)$ for some  field $F\subset \C$ and there is a valuation $v$ on $F$ such that $ v(\pm \tr(\bar{\rho}(\gamma)))<0$. Assume that $F$ is the smallest field containing the image. 
By the above discussion there is an affine point $\chi_{\bar{\rho}'}\in \bar{X}(Q)$ such that $\hat{p}(\chi_{\bar{\rho}'})=\chi_{\bar{\rho}}$.  Therefore, $\bar{\rho}':\pi_1^{orb}(Q) \rightarrow \PSL_2(F')$ for some $F'\subset \C$.  Again, assume that $F'$ is the smallest such field. Recall that $\hat{p}(\chi_{\bar{\rho}'}) = \chi_{\bar{\rho}' \circ p_*}$.  If $\gamma \in \pi_1(M)$ then 
\[ \chi_{\bar{\rho}}(\gamma) =  \chi_{\bar{\rho}' \circ p_*}(\gamma) = \chi_{\bar{\rho}'}(p_*(\gamma))\in F'.\]   We conclude that $F\subset F'$.

Let $\gamma\in \pi_1(M)$ be such that $v(\pm \tr(\bar{\rho}(\gamma)))<0$. Since $\bar{\rho}'(p_*(\gamma)))=\bar{\rho}(\gamma)$ we have   $v(\pm \tr(\bar{\rho}'(p_*(\gamma))))<0$ as well. 
The valuation $v$ on $F$ extends to a valuation $v'$ on $F'$.  Let  $T=T_{v',F'}$.
By Theorem~\ref{orbifoldTreeAction} there is a surface $F_Q$ in $Q$ dual to the action associated to $\chi_{\bar{\rho}'}$.  
By Lemma~\ref{lemma:liftedsurfaceisdetected} the surface, $F_Q$  dual to the action from $\chi_{\bar{\rho}'}$ lifts to a symmetric surface $F_M$ in $M$ which is detected by the induced action of $\pi_1(M)$ on $T$. 
The proof that this induced action of $\pi_1(M)$ on $T$  is the same as the action from the valuation corresponding to $\chi_{\bar{\rho}}$ follows similar to the above.
\end{proof}

The condition that the quotient orbifold has a flexible cusp should be thought of as a generic condition.  The only known knot complements in $S^3$ which cover orbifolds with rigid cusps are the dodecahedral knots \cite{MR1184399} and the figure-eight knot.

The following corollary  follows immediately because ideal points on a canonical component of the character variety of a one cusped manifold only detect surfaces with a single boundary slope.

\begin{cor}
Let $M$ be a compact, orientable, one cusped, hyperbolic 3-manifold.  If $G$ is a finite group of orientation preserving symmetries of $M$ such that the orbifold $M/G$ has a flexible cusp, then for every slope detected at an ideal point of $\bar{X}_0(M)$ there exists an essential surface with that slope that is fixed set wise by every element of $G$. 
\end{cor}

We now prove Corollary \ref{maincor}.
\begin{proof}[Proof of Corollary \ref{maincor}]
By \cite{MR2303551}, the ideal points of $X_0(M)$ detects at least two distinct strict boundary slopes.   Then the ideal points of $\bar{X}_0(M)$ detect at least two distinct boundary slopes of $M$.  Because $M$ is one cusped, each ideal point of $\bar{X}_0(M)$ only essential surfaces of a single slope.  Then $\bar{X}_0(M)$ has at least two ideal points that strongly detect distinct slopes.  Then by Theorem \ref{thm:maintheorem} these two ideal points detect essential surfaces which have distinct boundary slopes and are preserved by every orientation preserving symmetry of $M$.
\end{proof}

\section{Two-Bridge Knots}\label{examplesection}

Two-bridge knots are those knots admitting a projection with only two maxima and two minima as critical points.  To  each rational number $p/q$ with $q$ odd we associate a two-bridge knot; when $q$ is even this procedure produces a two component link.   There are various ways to see this association, we will demonstrate it through continued fractions  (see \cite{MR1959408}).   Continued fractions for $p/q$ can be used to determine surfaces in the knot complement by \cite{MR778125}. There are different conventions for the continued fractions.  We follow the treatment in \cite{MR778125} with exception of signs.  (Instead of the treatment below, they have a $-$ after each $a_i$ instead of a $+$. In the notation below, this  changes the sign on the even $a_i$ terms.)

\begin{figure}[ht!]
\begin{center}
\includegraphics[scale=0.25]{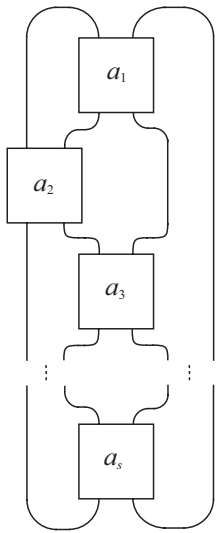} 
\caption{Two-bridge knot  }
\label{figure:twobridge}
\end{center}
\end{figure}
 
For each $p/q$ we consider  continued fractions that satisfy  
    \[ \frac{p}q  = r+ \frac{1}{a_1+ \cfrac{1}{a_2+\cfrac{1}{ \ddots+\cfrac{1}{a_s}}}}  \]
    for any integer $r$.  We write $p/q=r+[a_1,a_2,\dots, a_s]$.   
The knot $ K(p,q)$  is the boundary of the surface obtained by plumbing together $s$ bands in a row, the $i^{\mathrm{th}}$ band having $(-1)^{i-1}a_i$ half-twists, right-handed for a positive integer and left handed for a negative integer. (See Figure 2 of \cite[page 227]{MR778125}.)  Let $\Sigma[a_1,\dots, a_s]$ be the corresponding branched surface. Therefore, every two-bridge knot has a diagram as in Figure~\ref{figure:twobridge} where the numbers in each box indicate the number of half twists. The knot $K(p,q)$ is the same as the knot $K(p',q')$ exactly when $q=q'$ and $p'\equiv p^{\pm1} \pmod q.$

As mentioned, many continued fractions give the same $p/q$.
Hatcher and Thurston   showed that all $\pi_1$-injective surfaces in a two-bridge knot can be found by computing all the continued fraction expansions for $p/q$. 
Every continued fraction satisfying $|a_i|\geq 2$ for every $i$ corresponds to a branched surface in the knot complement from which some incompressible  surface can be seen using plumbing.  Surfaces obtained from different continued fractions are never isotopic. The boundary slopes of the surfaces correspond to the continued fraction in the following way (\cite{MR778125} Proposition 2).
For a given continued fraction $[a_1,a_2,\cdots,a_s]$ (with $|a_i| \geq 2$) the surfaces obtained from the corresponding branched surface all have slope 
\[ 2((n^+-n^-)-(n_0^+-n_0^-)),\]
 where $n^+$ and $n^-$ are the number of positive and negative terms respectfully in 
 \[ [a_1,-a_2,a_3,-a_4,\cdots,\pm a_s],\] and where $n_0^+$ and $n_0^-$ are the number of positive and negative terms respectfully in the unique continued fraction with all even terms after swapping the signs of the even numbered terms (i.e. swap the sign of the 2nd term, 4th term, etc.).   

Each two-bridge knot has a symmetry group of order 4 or order 8. The order 4 symmetry group is isomorphic to the Klein four group.  Every symmetry in this group acts trivially on the free homotopy classes of unoriented loops in $S^3-K(p,q)$ and therefore the induced action on the character variety is trivial. (See Section~\ref{section:symmandcanonical}.)
When the continued fraction $[a_1,a_2, \dots, a_s]$ representing the knot $K(p,q)$ is palindromic there is a symmetry $\sigma$ of the knot that turns the 4-plat upside down, and these knots have order 8 symmetry group. One can see this from the plumbing description of $K(p,q)$ above.   
The following lemma allows us to identify which surfaces in a two-bridge knot with this type of symmetry are preserved by $\sigma$. 


\begin{lemma}\label{lemma:twistsymmetries}
Let $K(p,q)$ be a two-bridge knot with palindromic twist region pattern, and let $\sigma$ be as above.
The action of $\sigma$ takes the branched surface $\Sigma[b_1, b_2 \dots, b_{s-1},  b_s]$  to $(-1)^{s+1} \Sigma[ b_s, b_{s-1}, \dots,b_2,  b_1]$.  

\end{lemma}

\begin{proof}
The action of $\sigma$ turns the 4-plat upside down.  For a branched surface given by $[b_1,b_2,\dots, b_s]$ this flips the twist regions with the conventions above, $b_i$ corresponds to $(-1)^{k+1}b_i$ twists which corresponds to the sign differences above. 
\end{proof}

Now we prove Corollary~\ref{cor:twobridge}, that every hyperbolic two-bridge knot contains at least two symmetric essential surfaces.

\begin{proof}[Proof of Corollary~\ref{cor:twobridge}]
First, suppose that $K$ is the  figure-eight knot.   Then $K$ has exactly one essential surface, $F_4$, with boundary slope $4$, and $K$ has exactly one essential surface, $F_{-4}$, with boundary slope $-4$. Because an orientation preserving homeomorphism of $K$ takes essential surfaces in $K$ to essential surfaces with the same boundary slope, every symmetry of $K$ takes $F_4$ to itself and $F_{-4}$ to itself.

Now, assume that  $K$ is a hyperbolic two-bridge knot other than the figure-eight knot. By Reid \cite{MR1099096}, $K$ is not  arithmetic.   As a result, $K$  admits no hidden symmetries by Reid and Walsh \cite[Theorem 3.1]{MR2443107}.   By Margulis  \cite{MR1090825}, since $K$ is not arithmetic,  there is a unique minimal element in the commensurability class of $S^3-K$.  The minimal orientable element is called the commenturator quotient.  Neumann and Reid \cite[Prop. 9.1]{MR1184416} show that a non-arithmetic    knot  has hidden symmetries if and only if this  commensurator quotient has a rigid cusp.   Therefore, the commensurator quotient of $S^3-K$ has a flexible cusp.  It follows that any orbifold that covers $S^3-K$ has a flexible cusp since the torus and pillowcase cannot be covered by a turnover. 
We conclude that the orbifold obtained from the quotient of the group of orientation preserving symmetries on $K$ has a flexible cusp.  
The statement now follows from  Corollary \ref{maincor}.
\end{proof}

\section{Double Twist Knots}

\begin{figure}[h]
\begin{center}
\includegraphics[scale=0.2]{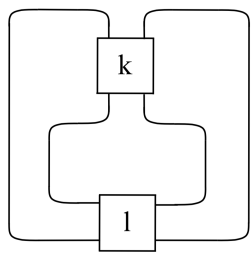} 
\caption{A $J(k,l)$ link  }
\label{doubleTwistKnot}
\end{center}
\end{figure}

A double twist link, often denoted $J(k,l)$, is a link with $k$ and $l$ half twist regions as seen in Figure \ref{doubleTwistKnot}. For $J(k,l)$ to be a knot $kl$ must be even; otherwise it is a two component link.  It was shown in \cite{MR2827003} that $\text{(P)SL}_2(\C)$ character variety of $J(k,l)$ consists of one component (of characters of irreducible representations) if $k\neq l$ and if $k=l$ then it has two components.  We will now concentrate on the case when $k=l$ as these are the only double twist knots which have a component of the character variety that does not consist of characters of abelian representations and is non-canonical.
Let $K_n=J(2n,2n)$ be a symmetric double twist knot.

We will explicitly calculate the slopes that correspond to essential surfaces in $S^3-K_n$ and then  prove that the symmetric slopes {\em are} detected on the canonical component, whereas the non-symmetric slopes {\em are not} detected on the canonical component.

The double twist knot $J(l,m)$ is equivalent to the two-bridge knot $K(p,q)$ with $\frac{p}{q}=\frac{l}{1-lm}$ in $\mathbb{Q}/\mathbb{Z}$ (see \cite{MR2827003}).
  Therefore, $K_n = K(2n, 1-4n^2)$  in two-bridge notation.  We will  calculate the slopes of the essential surfaces in the $K_n$ knots by using the continued fraction expansions of their two-bridge knot notation.  

\subsection{Continued Fraction Lemmas}
If a sequence of numbers $a_i,a_{i+1},\cdots a_{i+j}$ repeats $k$ times  in a continued fraction we will write this sequence as $(a_i,a_{i+1},\cdots, a_{i+j})_k$.  For example \[ [2,-2,2,-2]=[(2,-2)_2]\]  and 
\[ [5,3,2,3,2,7,2,3,2,3,2,3,2,3,2]=[5,(3,2)_2,7,(2,3)_4,2]=[5,(3,2)_2,7,2,(3,2)_4].\]

\begin{lemma}\label{lemmaNegativeContinuedFraction}
The following relationship holds  between negations of continued fractions
\[ [a_1,a_2,\cdots,a_s] = -[-a_1,-a_2,\cdots,-a_s].\]
\end{lemma}

\begin{proof}
Let $[a_1,a_2,\cdots,a_s]=\frac{p}{q}$.  Then 
\[
\frac{-p}{q} = \frac{-1}{a_1+ \frac{1}{a_2+ \cdots +\frac{1}{a_s} } } =\frac{1}{-a_1+ \frac{-1}{a_2+ \cdots +\frac{1}{a_s} } }= \frac{1}{-a_1+ \frac{1}{-a_2+ \frac{-1}{a_3 + \cdots +\frac{1}{a_s}} } }.
\] 
Continuing in this fashion, 
\[ \frac{-p}{q} =  \frac{1}{-a_1+ \frac{1}{-a_2+ \cdots +\frac{1}{-a_s} } }=[-a_1,-a_2,\cdots,-a_s].\]
\end{proof}

\begin{lemma}{\label{repeatfractionlemma1}}
For any positive integer $s$, 
\[ [(-2,2)_s] = \frac{-2s}{2s+1} \ \text{ and } \ [(2,-2)_s] = \frac{2s}{2s+1}.\]
\end{lemma}
\begin{proof}
This follows by induction on $s$.  Notice that $[-2,2]=\frac{-2}{3}$. 
Assuming that $[(-2,2)_s]= \frac{-2s}{2s+1}$, we see that 
\[ {[(-2,2)_{s+1}]=\frac{1}{-2+ \frac{1}{2+ \frac{-2s}{2s+1}}} }=\frac{1}{-2+\frac{2s+1}{2s+2}}= \frac{2s+2}{-2s-3}\]
as needed. The second assertion follows from  Lemma \ref{lemmaNegativeContinuedFraction}.
\end{proof}

In what follows, we extend the definition of our continued fraction notation to include real number entries. We treat any real number as if it were an integer with respect to its placement in the continued fraction.

\begin{lemma}{\label{repeatfractionlemma2}}
For any non-negative integer $k$ and real number $x$, 
\[ [(2,-2)_k, 2,x] = \frac{(2k+1)x+2k}{(2k+2)x+2k+1}.\]  
\end{lemma}

\begin{proof}
Direct computation shows that $[2,x] = \frac{1}{2+\frac{1}{x}}= \frac{x}{2x+1}$.  
We proceed by induction on $k$.   Assume that $[(2,-2)_k, 2,x] = \frac{(2k+1)x+2k}{(2k+2)x+2k+1}$.  An elementary computation shows that  
\[ [(2,-2)_{k+1}, 2,x]=\frac{1}{2+\frac{1}{-2+\frac{(2k+1)x+2k}{(2k+2)x+2k+1}}}= \frac{(2k+3)x+2k+2}{(2k+4)x+2k+3}\]
as needed. 
\end{proof}

\begin{lemma}\label{repeatfractioncompositionlemma}
Given any continued fraction of length $s$ and any integer $i$ such that $1<i<s$, 
\[ [a_1,a_2,\cdots,a_i,a_{i+1},\cdots,a_s] =[a_1,a_2,\cdots,a_{i-1},\frac{1}{[a_i,a_{i+1},\cdots,a_s]} ]. \]
\end{lemma}

\begin{proof}
Let $[a_1,a_2,\cdots,a_s]$ be a continued fraction.  Then 
\[ [a_1,a_2,\cdots,a_s]=\frac{1}{a_1+\frac{1}{a_2+\cdots+\frac{1}{a_{i-1}+\frac{1}{a_i+\cdots +\frac{1}{a_s}}} }}.\]
Note that \[  \frac{1}{a_i+\cdots +\frac{1}{a_s}}= [a_i,\cdots,a_s].\]  Therefore, \[[a_1,a_2,\cdots,a_s]=\frac{1}{a_1+\frac{1}{a_2+\cdots+\frac{1}{a_{i-1}+[a_i,\cdots,a_s]} }}= [a_1,a_2,\cdots,a_{i-1},\frac{1}{[a_i,a_{i+1},\cdots,a_s]} ].\]
\end{proof}

\subsection{Boundary Slopes of Symmetric Double Twist Knots  }

The main result of this section, which we will use in Section~\ref{examplesections} to prove Theorem \ref{mainresultTheoremdetectedslopesstrict}, is the following.

\begin{prop}\label{slopeOfEssentialSurfaces}
For $|n|>1$ the slopes of essential surfaces in the knot $K_n$ are exactly $0$, $-4n$, and $-8n+2$.  Surfaces of slope $-4n$ are not preserved by the symmetry that turns the 4-plat upside down. 
\end{prop}

 Proposition~\ref{slopeOfEssentialSurfaces}  will follow directly from Proposition~\ref{proposition:continuedfractions} which establishes that $0$, $-4n$, and $-8n+2$ are the only boundary slopes of $K_n$, and Lemma~\ref{lemma:-4nnotsymmetric} which shows that $-4n$ is not a symmetric slope. 

\begin{prop}\label{proposition:continuedfractions}
The continued fractions for the knot $K_n$ when     $n>1$ are 
\[ [2n,-2n],[2n-1,2,(-2,2)_{n-1}],[(-2,2)_{n-1},-2,-2n+1],[(-2,2)_{n-1},-3,(2,-2)_{n-1}]\]  and these correspond to branched surfaces that give slopes of $0$,$-4n$, $-4n$, and $-8n+2$ respectfully. \label{propcontinuedfrac}
\end{prop}

\begin{proof}

The rational number associated to the two-bridge knot $K_n$ is $\frac{2n}{4n^2-1}$. We will first show that these four continued fractions are expansions of this rational number (modulo $\Z$) and then show that there are no other continued fraction expansions which are associated to branched surfaces. 

The fact that $[2n,-2n]= \frac{2n}{4n^2-1}$ is an elementary computation.   This  is the unique continued fraction with all  even terms. We conclude that $n_0^+=2$ and $n_0^-=0$ since these are the number of positive and negative (respectfully) terms in $[2n,-(-2n)]$.  For this expansion $n^+=n_0^+$ and $n^-=n_0^-$ so the associated slope is  $2((n^+-n^-)-(n_0^+-n_0^-))=0$.

The second continued fraction is $[2n-1,2,(-2,2)_{n-1}]$.
From  Lemma \ref{lemmaNegativeContinuedFraction} we have 
\[ [2n-1,2,(-2,2)_{n-1}]=-[-2n+1,-2,(2,-2)_{n-1}].\]  By  Lemma \ref{repeatfractionlemma1} this is 
\[ \frac{-1}{(-2n+1)+\frac{1}{-2+[(2,-2)_{n-1}]}}= \frac{-1}{(-2n+1)+\frac{1}{-2+\frac{2n-2}{2n-1}}}=\frac{2n}{4n^2-1}\]
as needed. We immediately see that  $n^+=1$ and $n^-=1+2(n-1)=2n-1$.  This continued fraction corresponds to surfaces with slope $2((n^+-n^-)-(n_0^+-n_0^-))=-4n.$

The third continued fraction is $[(-2,2)_{n-1},-2,-2n+1]$. 
By Lemma~\ref{lemmaNegativeContinuedFraction} \[ [(-2,2)_{n-1},-2,-2n+1]= -[(2,-2)_{n-1},2,2n-1].\]  Therefore we will consider $[(2,-2)_{n-1},2,2n-1]$. By Lemma~\ref{repeatfractionlemma2}  
 \[ [(2,-2)_{n-1}, 2,2n-1] = \frac{(2n-1)^2+2n-2}{(2n)(2n-1)+2n-1}=1-\frac{2n}{4n^2-1}.\]  
 Therefore, the original continued fraction is $-1+(2n /(4n^2-1))$, which gives the correct fraction. 
 Here  $n^+=1$ and $n^-=2n-1$ so it corresponds to surfaces with slope $-4n$ as well.

The fourth continued fraction is $[(-2,2)_{n-1},-3,(2,-2)_{n-1}]$. Call this $\tfrac{p}q$.
From Lemma \ref{repeatfractioncompositionlemma}, $\tfrac{p}q=[(-2,2)_{n-2},-2,\frac{1}{[2,-3,(2,-2)_{n-1}]}].$
By Lemma~\ref{repeatfractionlemma1},  $[(2,-2)_{n-1}]=(2n-2)/(2n-1)$ and so 
\[
[2,-3,(2,-2)_{n-1}]=\frac{1}{2+\frac{1}{-3+[(2,-2)_{n-1}]}} = \frac{4n-1}{6n-1}.
\]
Therefore, $\tfrac{p}q=[(-2,2)_{n-2},-2,\frac{6n-1}{4n-1}]$. By Lemma \ref{lemmaNegativeContinuedFraction}, we can negate the entries and have that  $\tfrac{p}{q}= -[(2,-2)_{n-2},2,-\frac{6n-1}{4n-1}].$ Now we can apply Lemma~\ref{repeatfractionlemma2}, obtaining
\[ \tfrac{p}{q} = -\frac{(2n-3)\frac{-6n+1}{4n-1}+2n-2}{(2n-2)\frac{-6n+1}{4n-1}+2n-1} =-1+\frac{2n}{4n^2-1}. \]
For this continued fraction $n^+=0$ and $n^-=4(n-1)+1=4n-3$.  Thus it corresponds to surfaces with slope $-8n+2$.

It now suffices to show that these are the only continued fractions associated to branched surfaces for these knots. Every rational number has a unique continued fraction that contains all positive terms and does not end in a $1$.   For $n>1$ this continued fraction is $\frac{2n}{4n^2-1}=[1,2n-2,1,2n-1]$.  However, this does not correspond to a branched surface in the knot complement because it has terms with 1's. Since this continued fraction is of the form $[1,a,1,a-1]$, it follows from  \cite[Lemma 7]{MR2441951}  that there are   a total of four continued fractions that correspond to  surfaces in the knot complement,  and two of those correspond to  surfaces with the same boundary slopes (note that \cite{MR2441951} uses slightly different notation for continued fractions).
%
\end{proof}

\begin{lemma}\label{lemma:-4nnotsymmetric}
The symmetry $\sigma$  does not fix any surface of slope $-4n$. 
\end{lemma}

\begin{proof}
By  Proposition~\ref{proposition:continuedfractions} it suffices to analyze the action of $\sigma$ on the branched  surfaces.  By Lemma~\ref{lemma:twistsymmetries}, $\sigma$  fixes the continued fractions \[ [2n,-2n] \ \text{ and}  \ [(-2,2)_{n-1},-3, (2,-2)_{n-1}]\]   and swaps \[  [2n-1,2,(-1,2)_{n-1}] \ \text{ and } \ [(-2,2)_{n-1},-2,-2n+1].\]  It follows that no surface of slope $-4n$ is fixed by $\sigma$, because in Hatcher and Thurston's classification (\cite{MR778125} Theorem 1 part  d)) surfaces correspond to unique continued fractions. 
\end{proof}

\subsection{Detected Slopes of the $K_n$ Double Twist Knots} \label{examplesections}

The character varieties for the knots $K_n$ ($|n|>1$) all have exactly two components which contain characters of irreducible representations, by \cite{MR2827003}.  Let $X_0(n)$ denote the canonical component, and $X_1(n)$ the other component with $X(n)=X_0(n)\cup X_1(n)$;  the algebraic set  $X(n)$ is the Zariski closure of the irreducible characters. 
In this section we will determine which slopes are detected by the ideal points of the character variety of the $K_n$ knots.

We will use the following theorem of Ohtsuki from  \cite{MR1248091}. The ideal points in the statement are ideal points of components of the $\PSL_2(\C)$ character variety which contain characters of irreducible representations (see \cite[Section 1]{MR1248091}). 
\begin{thm}[Ohtsuki]\label{OhtsukiTheorem1}
There is a 1 to 1 correspondence between the ideal points of $X(K(p,q))$ and the elements of the set:

\[ \bigcup_{[n_j]}(\{ (k_1,\cdots,k_N) |k_i\in\mathbb{Z}/(n_i),k_i\neq 0,  \exists i \text{ such that } k_i\neq \frac{n_i}{2}   \}/ \sim ). \]  Here the union is taken over all continued fraction expressions for $\frac{p}{q}$ and the equivalence relation is generated by $(k_j)\sim(-k_j)$ and $(k_j)\sim((-1)^j k_j)$.
\end{thm}

 We now prove Theorem~\ref{mainresultTheoremdetectedslopesstrict} which we  restate for convenience.

\symmetricdoubletwistthm*
 
\begin{proof} 
By Proposition~\ref{slopeOfEssentialSurfaces} we know that the strict boundary slopes for the knot $K_n$ are $0$, $-4n$, and $-8n+2$.  As a consequence of Theorem \ref{OhtsukiTheorem1} all strict boundary slopes in two-bridge knots are detected by ideal points of the character variety. Specifically, by Ohtuski's theorem,   the continued fractions 
\[  [2n-1,2,(-2,2)_{n-1}] \text{ and }  [(-2,2)_{n-1},-2,-2n+1] \] both have exactly $n-1$ ideal points corresponding to them.

For completeness, we now show that a canonical component $X_0(M)$ always detects at least two boundary slopes, although this is well known.  
The evaluation function  $I_{\gamma}:X_0(M) \rightarrow \C$ defined by $I_{\gamma}(\rho)=\chi_{\rho}(\gamma)$ is non-constant on $X_0(M)$  for all $\gamma \in \pi_1(\partial M)$ \cite[Proposition 2]{MR735339}.  
As a result, the image  $A_0(M)$  of $X_0(M)$ under the eigenvalue map to the $A$-polynomial curve cannot correspond to a  factor of the type $f(L^aM^b)$ where $f(x)$ is a polynomial. Therefore the Newton polygon  associated to $A_0(M)$ is two-dimensional.  By  \cite{ccgls} the boundary slopes of the Newton polygon are the boundary slopes detected by the character variety, so any canonical component $X_0(M)$ detects at least two slopes.

It remains to show that  $-4n$ is not detected on the canonical component.   We first show that $S^3-K_n$ has a flexible cusp.  By \cite{MR3575575}, if $M$ is a non-arithmetic hyperbolic two-bridge link complement, then $M$ admits no hidden symmetries.  By \cite{MR1099096} the figure-eight is the only arithmetic  knot.  Therefore, for $|n|>1$ the complement of $K_n$ admits no hidden symmetries. It follows that the corresponding quotient orbifold   has a flexible cusp  by  \cite{MR1184416}.  The fact that  $S^3-K_n$ has a flexible cusp follows.  By Lemma~\ref{lemma:-4nnotsymmetric} no surface of  slope $-4n$ is  preserved under the symmetry $\sigma$ which turns the 4-plat upside down.   Therefore, by  Theorem~\ref{thm:maintheorem}, a canonical component  cannot detect any surface of slope $-4n$.  From the above, it must detect at least two boundary slopes, so it must detect slopes 0 and $-8n+2$.  Therefore, there must be surfaces of these slopes which are preserved by the full symmetry group of the manifold. 
\end{proof}

 \begin{remark}

We believe that the slopes detected on the non-canonical component $X_1(M)$ of the character variety are exactly $0$ and $-4n$.  By direct computation, using \cite{MR3425625}, we found the corners of the $A$-polynomial for the $K_n$ knot for $n$ up to $10$ to be $(0,12n-1),(2,12n-1),(1,4n),(3,8n-2),(2,0),(4,0)$. This combined with a calculation of the Newton polygon of the canonical component of the character variety  also based on \cite{MR3425625}  shows that the Newton polygon for the non-canonical component of the $A$-polynomial for these knots has corners at $(0,8n-2),(1,8n-2),(1,0),(2,0)$.
\end{remark}



\bibliographystyle{amsplain}
\bibliography{myrefs.bib}

\end{document}